\documentclass[a4paper, 12pt]{article}
\usepackage{tgheros}
\usepackage{amsthm}
\newtheorem{thm}{Theorem}
\newtheorem{cor}[thm]{Corollary}

\newtheorem{lemma}[thm]{Lemma}

\usepackage{amssymb,amsmath}

\newtheorem*{theorem*}{Theorem}

\DeclareMathOperator{\F}{\mathbb{F}}

\begin{document}
\baselineskip=16.5pt
\parskip=15pt

\begin{center}
\section*{Invariant Rational Functions, Linear Fractional Transformations and Irreducible 
Polynomials over Finite Fields}

{\large 
Rod Gow and Gary McGuire\footnote{email {\tt gary.mcguire@ucd.ie}}
 \\ { \ }\\
School of Mathematics and Statistics\\
University College Dublin\\
Ireland}
\end{center}

\bigskip

 \subsection*{Abstract}
 
 For a  subgroup of $PGL(2,q)$ we show how some irreducible polynomials over $\mathbb{F}_q$ arise
 from the field of invariant rational functions.
 The proofs rely on two actions of $PGL(2,F)$, one  on the projective line over  a field $F$
 and the other on the rational function field $F(x)$.
 The invariant functions in $F(x)$  are used to 
 show that regular patterns exist in the factorization of certain
 polynomials into irreducible polynomials.
We use some results about group actions and the orbit polynomial, whose proofs
are included.
An unusual connection to the conjugacy classes of $PGL(2,q)$ is shown.
At the end of the paper we present an alternative approach, using Lang's theorem on algebraic groups.
 
 \newpage

\section{Introduction}
  
This work began with seeking an explanation for  patterns in certain polynomial 
factorizations over finite fields. For example,  
factoring $T^{20}-T^{19}+T+1$ over $\F_{19}$ into irreducible factors gives the following factors:
 \begin{align*}
 T^4 + 6T^3 &-6T^2 + 13T +1\cr
 T^4 + 9T^3 &-6T^2 + 10T + 1\cr
 T^4 + 12T^3 &-6T^2 + 7T + 1\cr
 T^4 + 14T^3 &-6T^2 + 5T + 1\cr
 T^4 + 15T^3 & -6T^2 + 4T + 1\cr
 \end{align*}

Each of these factors has the form $T^4-\lambda T^3-6T^2+\lambda T+1$.
We wondered why this pattern appears in every factor, and if it was a coincidence.

We will prove the following theorem which explains the pattern in this example.
This statement is a summary; for a more precise statement see Theorem  \ref{main3b} and
Corollary \ref{main1} later, 
and for the definition of the orbit polynomial and linear 1-parameter family see Section \ref{orbitp}.

 \begin{theorem*}
 Let $G$ be a cyclic subgroup 
 of $PGL(2,q)$ of order $r>2$ generated by $s(x)=\frac{ax+b}{cx+d}$,
 where  $r$ divides $q+1$.
 Then the polynomial $cT^{q+1}+dT^q-aT-b $ factors into 
 irreducible factors that all come from a linear 1-parameter family of polynomials of degree $r$ in $ \F_q [T]$.
 This family arises from the orbit polynomial of $G$ by specialization.
 \end{theorem*}

 This paper is organized as follows.
 Section 2 presents a quick summary of some background.
 In Section \ref{pglactp1} we  present results about the orbits and fixed points of the action of
 a subgroup of $PGL(2,F)$ on $\mathbb{P}^1 (F)$,
 the projective line over $F$, where $F$ is a field.
 These are well-known topics, however we found it difficult to pinpoint in the
 literature the results we require so we have proved them here,
which also  makes the exposition self-contained.
We do not claim that any results in Section \ref{pglactp1} are new.
  In Section \ref{orbitp} we consider the orbit polynomial of a finite subgroup of  $PGL(2,F)$  and make some 
  basic observations. 
 We also briefly mention the connection to the Riemann-Hurwitz formula  in Section \ref{orbitp},
 although we do not use it.
 
 Section   \ref{xqplus1}   considers
 the cyclic subgroup of $PGL(2,q)$ generated by 
 $ x\mapsto \frac{ax+b}{cx+d}$, and its relation to the Frobenius automorphism.
 This naturally results in polynomials of the form $cT^{q+1}+dT^q-aT-b \in \F_q [T]$.
 There are many papers that discuss various aspects of these polynomials.
 Previous work on the factorization of such polynomials into irreducible factors 
  has been done by Stichtenoth and Topuzo\u{g}lu \cite{ST}.
 Our new contribution is the connection to the field of invariant rational functions,
 which implies a relationship between the  factors and the orbit polynomial;
 see the theorem stated above.
  The theorem shows that the form of the orbit polynomial
 implies certain patterns in the form of the  irreducible factors.
 The proof uses results on fixed points proved in Section  \ref{pglactp1}.

  It is well known that the maximal order of an element of  $PGL(2,q)$ is $q+1$.
  Section  \ref{xqplus1}  also
  focuses on cyclic subgroups of $PGL(2,q)$ whose order divides $q+1$.
  For such subgroups we have more precise information.
One theorem we will prove is the following (see Theorem \ref{orderq1}),
whose proof relies on Lemma \ref{nofixed} in Section  \ref{pglactp1}.

\begin{theorem*}
Let $G$ be a cyclic subgroup of $PGL(2,q)$ of order $q+1$.
Let $f(x)/g(x)$ be a generator for the field of $G$-invariant rational functions,
where $\deg(f)=|G|$ and $\deg(g)<|G|$.
Then $f(x)$ has the form $cx^{q+1}+dx^q-ax-b$.
\end{theorem*}

Hence polynomials of the form $cx^{q+1}+dx^q-ax-b$ arise naturally in the context of
invariant rational functions, as well as relating $PGL(2,q)$ action to Frobenius action.

In Section \ref{degorderG} we shall consider a family of 
generators for the field of invariant rational functions
$(f(x)-\lambda g(x))/g(x)$ for $\lambda\in \F_q$.
We will determine how these numerators factor into
irreducible factors using  the interaction of the $G$-action on the roots and the Frobenius action.
The $G$-action shows that the roots lie in a  $G$-orbit, but these roots can
 lie in different extensions of $\F_q$ for different values of $\lambda$.
We will prove the following (see Theorem \ref{falpha1}) using
a general principle of two commuting group actions,
 which we prove in Lemma \ref{comm_action2} in Section  \ref{pglactp1}.

\begin{theorem*}
Let $G$ be a subgroup of $PGL(2,q)$.
Let $\Phi(x)=f(x)/g(x)$ be a generator for the field of $G$-invariant rational functions,
where $\deg(f)=|G|$ and $\deg(g)<|G|$.
If $\alpha\in\overline{\F_q}$ and $\Phi(\alpha)\in \F_q$ 
then $f(T)-\Phi(\alpha) g(T)$ factors into irreducible factors of degree $r$,
where $r$ is the degree of the minimal polynomial of $\alpha$.
\end{theorem*}

There are $q+1$ points on the projective line over $\F_q$.
There are $q+1$ or $q+2$ conjugacy classes in $PGL(2,q)$, according as $q$ is even or odd.
Section 7 explains a natural and also unusual correspondence between conjugacy classes in $PGL(2,q)$,
and the elements of the projective line over $\F_q$, which arises as a result of our 
investigations.
Particular attention is paid to elements of order 2, which is in keeping with the importance of 
the centralizer of an involution in finite group theory.

 We remark that although the finite
 subgroups of $PGL(2,F)$ have been classified,
 our results are abstract and do not use the classification.
 In Section \ref{s5}  we present an application of our lemmas - a short proof that $S_5$ is not a subgroup of 
 $PGL(2,F)$ when $F$ does not have characteristic 5. 
 In Section \ref{lang}  we present an alternative proof of our main working result, Theorem \ref{main3b}.
 This proof uses Lang's theorem on algebraic groups, and yields some additional information
 as well as providing another context.

 \section{Background}
 
 In this section we present some background for the later sections.
  In this paper $\overline{F}$ will denote an algebraic closure of a field $F$.
  We use the notation $\F_q$ for a finite field with $q$ elements,
  where $q$ will always be a power of a prime number $p$.
 
 \subsection{Projective Space and $PGL(2,F)$}
 
 The general linear group $GL(2,F)$ is the group of all invertible $2$-by-$2$
 matrices with entries in $F$. The projective general linear group $PGL(2,F)$
 is the quotient of $GL(2,F)$ by the subgroup consisting of all the nonzero scalar multiples of the identity
 (its centre).

Let $\mathbb{P}^1(F)$ denote the projective line over a field $F$.
 The points of $\mathbb{P}^1(\F_q)$ are the $(q^2-1)/(q-1)=q+1$
one-dimensional subspaces of $\F_q^2$.  
The 1-subspace containing $(x,y)$ is defined by the value
$x/y$, with the convention that $x/y=\infty $ if
$y=0$, $x\not= 0$, and this provides a bijective correspondence
between 1-subspaces (points of $\mathbb{P}^1(\F_q)$) and elements of
$\F_q\cup \{\infty \}$.  We may think of the elements of
$\F_q\cup \{\infty \}$ as the slopes of the 1-subspaces.

In this way we usually identify $\mathbb{P}^1(\F_q)$ with
$\F_q\cup \{\infty \}$.
This identification induces the following action of $PGL(2,q)$.
For an element ${a\ b\choose c\ d}$ of $GL(2,q)$ the
corresponding element of $PGL(2,q)$ is the map
$z\mapsto {az+b\over cz+d}$, with the conventions that
${a\infty +b\over c\infty +d}={a\over c}$ if $c\not= 0$,
${a\infty +b\over c\infty +d}=\infty$ if $a\not= 0$, $c=0$,
and $a/0=\infty$ if $a\not= 0$.

We then have
$$PGL(2,q)=\bigg\{z\mapsto {az+b\over cz+d}: \ ad-bc\not= 0 \bigg\} $$
 as a group of linear fractional (or M\"obius) transformations on $\mathbb{P}^1(\F_q)$.
 The order of $PGL(2,q)$ is $q(q-1)(q+1)$.

 We could also let $PGL(2,q)$ act on the projective line over a finite extension field,
  $\mathbb{P}^1(\F_{q^r}) $, or over the algebraic closure $\mathbb{P}^1(\overline{\F_q}) $,
  by the same formula $z\mapsto {az+b\over cz+d}$.
 This will be relevant in later sections.

 \subsection{Function Fields}
 
 Let $F$ be a field,  let $x$ be an indeterminate over $F$, and let $F(x)$ be the 
 field of rational functions in $x$.

  The degree of a rational function is defined to 
 be the maximum of the degrees of the numerator and denominator.
We always assume the numerator and denominator are relatively prime, as polynomials.

 It is well known (see \cite{vdw} for example)
 that if $r(x)$ is a nonconstant rational function in $F(x)$, then
 the degree of the field extension $[F(x):F(r(x))]$ is equal to the degree of $r(x)$.
 It follows that $F(x)=F(r(x))$ if and only if $r(x)$ has degree 1.
 
 It then follows that the full group of $F$-automorphisms of $F(x)$  is  $PGL(2,F)$,
 which acts by mapping $r(x)\in F(x)$ to $r(\frac{ax+b}{cx+d})$, where 
 $a,b,c,d \in F$ and the matrix 
 ${a\ b\choose c\ d}$ is invertible.
 Therefore, if $G$ is a finite group of $F$-automorphisms of $F(x)$,
 we may assume that $G$ is a subgroup of $PGL(2,F)$.
 
 All automorphisms in this context are assumed to be $F$-automorphisms.
 Let $G$ be a finite group of automorphisms of $F(x)$, i.e., a finite subgroup of $PGL(2,F)$.
 Let $F(x)^G$ denote the subfield of rational functions that are fixed by every element of $G$, i.e.,
 \[
 F(x)^G=\{ r\in F(x) : r \circ g = r \text{ for all } g\in G\}.
 \]
 By L\"uroth's Theorem there exists a rational function $\Phi(x)$ such that
 $F(x)^G = F(\Phi(x))$.
 This $\Phi(x)$ may be called a generator for the invariant functions of $G$.
 By the remarks so far, $\frac{a\Phi +b}{c\Phi +d}$ is also a generator, 
 where ${a\ b\choose c\ d}\in GL(2,F)$, and all
 other generators have this form.
 
 By Galois theory \cite{A} we know that $F(x):F(x)^G$ is a Galois extension with
Galois group $G$, and  $[F(x):F(x)^G]=|G|$.
 It follows that the degree of a generator for the $G$-invariant rational functions is $|G|$,
 and that a nonconstant $G$-invariant rational function cannot have degree less than $|G|$.
 
 As far as we are aware, 
 an explicit   generator for $\F_q(x)^G$ when $G=PGL(2,q)$ has been calculated first in 
 Rivoire \cite{R} who also deals with most subgroups.
One can also find this calculated in the appendix by Serre  to Abhyankar \cite{Ab},
in Chu-Kang-Tan \cite{CKT},
in Gutierrez-Sevilla  \cite{GS} and in Bluher \cite{B}.
 One such generator for $PGL(2,q)$ is
 \begin{equation}\label{pglgen}
\Phi(x)= \frac{(1+(x^q-x)^{q-1})^{q+1}}{(x^q-x)^{q^2-q}}.
 \end{equation}
 An explicit  generator for $\F_q(x)^G$ when $G=PSL(2,q)$ was first found by Dickson \cite{D}.

 \subsection{Orbits and Group Actions}
 
 Let $G$ be a finite group acting on a set $\Omega$, which may be finite or infinite.
 For $g\in G$ and $x\in \Omega$ we denote the image of $x$ under $g$ by $g(x)$.
 An orbit for the action is a set of the type
 ${\cal O}_x=\{ g(x) : g\in G\}$.
 An orbit is called regular if $|{\cal O}_x|=|G|$, and is called non-regular (or short)  if $|{\cal O}_x|<|G|$.
 
 The stabilizer of $x\in \Omega$ is $G_x=\{ g \in G : g(x)=x\}$ and 
 the Orbit-Stabilizer theorem states that $[G:G_x]=|{\cal O}_x|$.
 It follows that $|G_x|=1$ iff  ${\cal O}_x$ is a regular orbit.
  When there is only one orbit we say that $G$ acts transitively on $\Omega$.
 If $G$ acts transitively, and if $|G_x|=1$ for all $x\in\Omega$, we say that $G$ acts regularly on $\Omega$
 (every orbit is a regular orbit).

 We say that $z\in \Omega$ is a fixed point of $g\in G$ if $g(z)=z$.
 If $z$ is a fixed point of $g$, 
 then $h(z)$ is a fixed point of $hgh^{-1}$.
 This fact is often used to move a  fixed point by conjugation.

\subsection{Linear Algebra}\label{la}

In this subsection we will assume $F=\F_q$ and we spell out some details
from linear algebra that we will use later, see for example the proofs of Lemma \ref{two_fixed},
Lemma \ref{Sylow}, and Theorem \ref{numberofr}.

 Recall that any matrix $A\in GL(2,q)$ has an image $s$ in $PGL(2,q)$,
 and all  nonzero scalar multiples of $A$ will have
the same image in $PGL(2,q)$.
Conversely, an element of  $PGL(2,q)$ has $q-1$ pre-images in $GL(2,q)$,
which are invertible matrices that are all scalar multiples of each other.
We summarize the different behaviours of eigenvalues of $A$, and
how this relates to the order of $s$ and to the fixed points of $s$.

By elementary linear algebra, exactly one of the following holds.
This is essentially the Jordan canonical form for  $GL(2,q)$.
\begin{enumerate}
\item $A$ has two distinct eigenvalues $\alpha_1, \alpha_2$  in $\F_q$. 
In this case $A$ is similar to  ${\alpha_1 \ 0\choose 0\ \alpha_2}$
over $\F_q$, and the order of $A$ is a divisor  of $q-1$.
Therefore $s$ is conjugate to
the image of ${\alpha_1 \ 0\choose 0\ \alpha_2}$ in $PGL(2,q)$,
which is $x\mapsto \alpha x$ where 
$\alpha=\alpha_1 \alpha_2^{-1}\not=1$.
This has order dividing $q-1$ and obviously has fixed points $0$ and $\infty$
(and see  Lemma \ref{two_fixed}).
\item $A$ has two equal eigenvalues $\alpha, \alpha$ in $\F_q$.
In this case $A$ is similar to  ${\alpha \ 1\choose 0\ \alpha}$ and the order
of $A$ is $pd$ where $d$ divides $q-1$, because 
${\alpha \ 1\choose 0\ \alpha}^n={\alpha^n \ n\alpha\choose 0\ \ \alpha^n}$.
Note that $\alpha =1$ if and only if $d=1$. 
Therefore $s$ is conjugate to the image of
${\alpha \ 1\choose 0\ \alpha}$ in $PGL(2,q)$, which is $x\mapsto  x+\alpha^{-1}$.
This has order  $p$  and has one fixed point, $\infty$
(and see  Lemma \ref{two_fixed}).
\item $A$ has two eigenvalues in $\F_{q^2}\setminus \F_q$, which are necessarily distinct.
In this case the characteristic polynomial of $A$ is irreducible over $\F_q$.
The order of $A$ is a divisor of $q^2-1$ that is not a divisor of $q-1$, 
and the order of $s$ is a divisor of $q+1$.
For information about the fixed points, see Lemma \ref{two_fixed} and Lemma \ref{nofixed}.
\end{enumerate}

When $q$ is odd, elements of order 2 in $PGL(2,q)$ come in two types, some in 
the first category and some in the third category.
In Section 7 we will expand on this distinction.

 \section{The action of  $PGL(2,F)$
 on the projective line $\mathbb{P}^1(F)$}\label{pglactp1}
 
 In this section we will prove some general results about $PGL(2,F)$
 acting on the projective line over $F$.
 If $F$ is not algebraically closed, then $PGL(2,F)$ also acts on the 
 projective line over $\overline{F}$, and obviously $PGL(2,F)$ is a subgroup 
 of $PGL(2,\overline{F})$.
 Therefore we will sometimes assume in this section that $F$ is an algebraically closed field
 of characteristic $p>0$.
 We do not claim that any results in Section \ref{pglactp1} are new.
 
 \subsection{Fixed Points}
 
 For a linear fractional transformation $s\in PGL(2,F)$,
 recall that $z\in \overline{F}$ is a fixed point of $s$ if $s(z)=z$.
 We say that a subgroup $G$ of $PGL(2,F)$ fixes $z$ if $z$ is a fixed point for  every element of $G$.
 Our first lemma can also be found as Theorem 11.14 in \cite{HKT}.

 \begin{lemma}\label{two_fixed}
 Let $F$ be an algebraically closed field of characteristic $p>0$.
 Any element of $PGL(2,F)$ has at most two fixed points in its action on 
 the projective line $\mathbb{P}^1(F)$. 
 All elements of $PGL(2,F)$ have exactly two fixed points except the elements of order $p$
 which have exactly one fixed point.
 \end{lemma}
 
 \begin{proof}
 Let $s(x)=\frac{ax+b}{cx+d}$ be an element of $PGL(2,F)$
 and suppose that $s(x)$ has a fixed point  $z\in \mathbb{P}^1(F)$.
 Then $az+b=z(cz+d)$ which means that fixed points are the roots of the quadratic 
 polynomial $cT^2+(d-a)T-b$.
 This quadratic will have two distinct roots unless either $c=0$ or $(d-a)^2+4bc=0$, if $p$ is odd.
 In characteristic 2 this changes to $c=0$ or $a=d$.
 
 If $c=0$ then $s(x)$ has the form $ax+b$, which fixes $\infty$. 
 If $a\not= 1$ then $b/(1-a)$ is also a fixed point, and $s(x)$ has two distinct fixed points.
 If $a=1$ then the only fixed point is $\infty$.
 In this case $s(x)=x+b$ has order $p$.
 
 If $(d-a)^2+4bc=0$ in odd characteristic, or $a=d$ in even characteristic,
 then the characteristic polynomial of the 
 matrix $A={a\ b\choose c\ d}$ in $GL(2,F)$ has one root of algebraic multiplicity two.
 The same is true for any conjugate of $A$, since conjugates have the 
 same characteristic polynomial. We can always conjugate $A$ in $GL(2,F)$  to make
 the bottom left entry 0. It follows that a conjugate of $s$ has fixed point  $\infty$, 
 and no other fixed points. 
 As in the previous paragraph, this conjugate then must have order $p$, and so does $s$.
 \end{proof}
 
 {\bf Remark.} For $PGL(2,q)$, the fixed points are either defined over $\F_q$ or $\F_{q^2}$,
 depending on whether the quadratic polynomial in the proof has roots
 in $\F_q$ or $\F_{q^2}$.

 Next we shall study subgroups that do have a fixed point.

\begin{lemma} \label{cyclic_stabilizer}
Let $G$ be a finite subgroup of $PGL(2,F)$, where $F$ is any field.
Suppose that $G$ fixes a point in its action on the projective line $\mathbb{P}^1(F)$. 
\begin{enumerate}
\item If $|G|$ is not divisible by the characteristic of $F$, then $G$ is cyclic. 
\item If $|G|$ is divisible by $p=char(F)$, then 
$G$ has an elementary abelian normal $p$-subgroup with cyclic quotient of order coprime
to $p$.
\end{enumerate}
\end{lemma}

\begin{proof}
$PGL(2,F)$ acts transitively on $\mathbb{P}^1(F)$, and hence, replacing $G$ by a conjugate under
$PGL(2,F)$ if necessary, it suffices to assume that $G$ fixes $\infty$. The subgroup of
$PGL(2,F)$ that fixes $\infty$ consists of all transformations $s$ of the form
\[
s(x)=ax+b,
\]
where $x\in \mathbb{P}^1(F)$, $a$ runs over the multiplicative group $F^*$ of non-zero elements of $F$ and $b$ runs over all elements of $F$. 

Now if we have a second transformation $t$ of the form $t(x)=a'x+b'$ that fixes $\infty$, we calculate that
\[
ts(x)=aa'(x)+a'b+b'.
\]
It follows that the map $\sigma:G\to F^*$ given by $\sigma(s)=a$ is a homomorphism. 

Consider the kernel of $\sigma$. This consists of all transformations $s$ of the form $s(x)=x+b$, for some $b$ in $F$. Now if $b\neq 0$ and $F$ has characteristic 0, any such transformation has infinite order and hence cannot belong to the finite group $G$. Hence $\sigma$ is injective in this case.

If $b\neq 0$ and the characteristic of $F$ is a prime $p$, $s$ has order $p$. Furthermore, it is easy to see that any two transformations having the form $x\mapsto x+c$ for some $c$ in $F$ commute. Thus the kernel is an elementary abelian $p$-subgroup of $G$. 

We also observe that $\sigma(G)$ is a finite subgroup of the multiplicative group of a field and is hence cyclic of order prime to $p$. Thus $G$ has a normal elementary abelian $p$-subgroup with cyclic quotient of order prime to $p$, as required.
\end{proof}

\subsection{Non-regular Orbits}\label{nro}

Next shall classify the non-regular orbits in the action of a finite subgroup of $PGL(2,F)$ on the projective line.

\begin{lemma} \label{three_non_regular_orbits}
Let $G$ be a finite subgroup of $PGL(2,F)$, where $F$ is an algebraically closed field. Then in its action on the projective line $\mathbb{P}^1(F)$, $G$ has at most three orbits that
are not regular. Moreover, if $G$ has exactly three non-regular orbits and if $|G|$ is not divisible by the characteristic of $F$, either $G$ has a normal cyclic subgroup of index $2$ or $|G|$ is one of $12$, $24$ or $60$. In the former case, $G$ is not abelian unless it is an elementary abelian group of order $4$.
\end{lemma}

\begin{proof}
Suppose if possible that $G$ has at least four regular orbits and let $\Omega_i$, $1\leq i\leq 4$, be four such orbits. Let $\Omega$ be the union of these $\Omega_i$. Then $\Omega$ is a finite subset of $\mathbb{P}^1(F)$ on which $G$ acts with exactly four orbits. Let $G_i$ be the stabilizer of a given point in $\Omega_i$, $1\leq i\leq 4$. Note that we have $|G_i|>1$ for all $i$, since the $\Omega_i$ are not regular. 

Let $g$ be any element of $G$ and let $\chi(g)$ denote the number of fixed points of $g$ in its action on $\Omega$. We have $0\leq \chi(g)\leq 2$ if $g$ is not the identity, since such a $g$ fixes at most two points of $\mathbb{P}^1(F)$. By the well known orbit counting lemma, sometimes said to be due to Burnside, we have
\[
4|G|=\sum_{g\in G} \chi(g)\leq |\Omega|+2(|G|-1).
\]
Thus since 
\[
|\Omega|=|\Omega_1|+\cdots+|\Omega_4|=\frac{|G|}{|G_1|}+\cdots+\frac{|G|}{|G_4|},
\]
we obtain the inequality 
\[
2\leq \frac{1}{|G_1|}+\cdots+\frac{1}{|G_4|}-\frac{2}{|G|}.
\]
This is clearly impossible, since each fraction $1/|G_i|$ is at most 1/2.
We deduce that $G$ has at most three non-regular orbits on $\mathbb{P}^1(F)$. 

We now consider the case that $G$ has exactly three non-regular orbits, $\Omega_i$, $1\leq i\leq 3$, with one-point stabilizers $G_i$, where we choose notation so that $|G_1|\leq |G_2| \leq |G_3|$.
As before, we let $\chi(g)$ denote the number of points of $\Omega$ (the union of the $\Omega_i$) fixed by $g\in G$. We claim that under the hypothesis that $|G|$ is not divisible by the characteristic of $F$, $\chi(g)=2$ if $g\neq 1$. 

For, as we have already proved, $g\neq 1$ fixes exactly two points of $\mathbb{P}^1(F)$, unless it is unipotent, in which case it fixes exactly one. However, $g$ cannot be unipotent, as it implies $g$ has infinite order or order $p$, where $p$ is the characteristic of $F$, inadmissible by hypothesis on $G$ (compare with the proof of Lemma \ref{cyclic_stabilizer}).
Thus any non-identity element of $G$ has exactly two fixed points on $\mathbb{P}^1(F)$. 

Suppose if possible that $\chi(g)\neq 2$. Then there is a point $\omega$, say, of $\mathbb{P}^1(F)$ fixed by $g$ that is not in $\Omega$. Consider the $G$-orbit containing $\omega$. It is not regular, as $g$ is in the stabilizer subgroup of $\omega$, and hence is contained in $\Omega$. This is a contradiction, and we see that $\chi(g)=2$ for all non-identity $g$.

The orbit counting lemma now yields that
\[
3|G|=|\Omega|+2(|G|-1)
\]
and hence
\[
1=\frac{1}{|G_1|}+\frac{1}{|G_2| }+\frac{1}{|G_3|}-\frac{2}{|G|}.
\]
Given the numbering that $|G_1|\leq |G_2|  \leq |G_3|$, we cannot have $|G_1|\geq 3$ and thus $|G_1|=2$. This leads to
\[
\frac{1}{2}=\frac{1}{|G_2|}+\frac{1}{|G_3|}-\frac{2}{|G|}.
\]

It is clear from this inequality that $|G_2|=2$ or 3. If $|G_2|=2$, we obtain $2|G_3|=|G|$.
Thus $G_3$ has index 2 in $G$ and is hence normal. Furthermore, as $G_3$ fixes a point in
$\Omega_3$, it is cyclic by Lemma \ref{cyclic_stabilizer}. Thus $G$ has a normal cyclic subgroup of index 2 in this case.

Consider next the other possibility that $|G_2|=3$. Then we obtain
\[
\frac{1}{6}=\frac{1}{|G_3|}-\frac{2}{|G|}.
\]
From this equality, it is clear that $3\leq |G_3|\leq 5$ and the three feasible values for
$|G_3|$ give the three outcomes $|G|=12, 24$ or 60.

Finally, suppose that $G$ is an abelian group that has exactly three non-regular orbits. 
Note that $G$ cannot have order 2 in this case, since it would imply that a generator of $G$ fixed three points.
We have proved that in all cases there is an orbit $\Omega_1$ whose one point stabilizer $G_1$ has order 2. Let the element $t$ of order two generate $G_1$ and let $t$ fix the point $\omega$ in $\Omega_1$. Since $G$ is assumed to be abelian, for any $g$ in $G$, $g\omega$ is also fixed by $t$ and this point is in $\Omega_1$. Thus since $t$ fixes at most two points of $\Omega_1$, it follows that there is a $G$-orbit of size at most two, consisting of fixed points of $t$, contained in the $G$-orbit $\Omega_1$. Thus, since $G>2$, it follows that
$|\Omega_1|=2$ and $|G|=4$. To complete the proof, we show that $G$ cannot be cyclic of order 4. For suppose that this is the case. Then, since there must be three $G$-orbits where the point-stabilizer has order two, and $G$ contains a unique subgroup of order two, the unique element of order two in $G$ fixes three points, which we know cannot happen. 
\end{proof}

{\bf Remark.}
The exceptional case where $G$ is an elementary abelian group of order 4 can certainly occur, as the following example indicates. Suppose that $F$ contains a primitive 4-th root of unity, $j$, say. Then the three commuting fractional transformations of order 2 defined by
\[
s(x)=-x,\quad t(x)=\frac{1}{x},\quad st(x)=-\frac{1}{x}
\]
act on $\mathbb{P}^1(F)$ and generate an elementary abelian group $G$ of order 4. $G$ has exactly three non-regular orbits, representative elements for these orbits being $\infty$, $-1$ and $j$.

{\bf Remark.} The group theoretic information contained in our proof enables us to identify the 
groups concerned, using the group itself acting on cosets of subgroups.
In the cases $|G|=12, 24$ or 60, it can be shown that $G$ is isomorphic to $A_4$, $S_4$ or $A_5$
respectively. We omit the details.

{\bf Remark.} The previous lemma can also be proved using the Riemann-Hurwitz formula,
in a similar manner to Theorem 11.56 in \cite{HKT}.

Next we focus on the case that $|G|$ is divisible by the characteristic.
In this case the Sylow subgroups are important.

 \begin{lemma} \label{Sylow}

Let $G$ be a finite subgroup of $PGL(2,\overline{\mathbb{F}}_q)$, where $p$ divides $|G|$. Let $Q$ be a Sylow $p$-subgroup of $G$
and let $N_G(Q)$ denote its normalizer in $G$. 
Then the following are true:

\begin{enumerate}
\item  $Q$ has a unique fixed point on $\mathbb{P}^1(\overline{\mathbb{F}}_q)$.

\item $Q$ is elementary abelian.

\item The centralizer in $G$ of any non-identity element of $Q$ is $Q$.

\item Any two different Sylow $p$-subgroups of $G$ have trivial intersection. 

\item Let $M$ denote the number of elements of order $p$ in $G$. Then 
\[
M=|G:N_G(Q)|(|Q|-1)
\]
and this number is congruent to $-1$ modulo $p$. 

\item $N_G(Q)$ fixes the unique fixed point of $Q$ on $\mathbb{P}^1(\overline{\mathbb{F}}_q)$ and it is the stabilizer in $G$ of this point.
\end{enumerate}

\end{lemma}

\begin{proof} For simplicity of notation, we let $F$ denote $\overline{\mathbb{F}}_q$.
The elements of $G$ consist of linear fractional transformations $s$ that act according to the formula $s(x)=(ax+b)/(cx+d)$, where $a$, $b$, $c$, $d$ are elements of $F$. These coefficients are algebraic over $\mathbb{F}_q$ and hence lie in some finite extension, $E$, say, of $\mathbb{F}_q$. Thus since $G$ is finite, enlarging $E$ if necessary, we may assume that $G$ is a subgroup of the finite group $PGL(2,E)$.

1. The Sylow subgroup $Q$ acts on the projective line $\mathbb{P}^1(E)$ over $E$, 
 which we may consider to be a subset of $\mathbb{P}^1(F)$. 
 The $Q$-orbits on $\mathbb{P}^1(E)$ have size a power of $p$.
 Note that  $|\mathbb{P}^1(E)|=q^s+1$ for some positive integer $s$.
Since $\mathbb{P}^1(E)$ is the disjoint union of $Q$-orbits, and these orbits have size a power of $p$,  
there must be at least one $Q$-orbit of size 1. This shows that there is a fixed point of $Q$ on 
$\mathbb{P}^1(E)$, and thus  on $\mathbb{P}^1(F)$. 
That the fixed point is unique follows from Lemma \ref{two_fixed}.

2.  It now follows from Lemma  \ref{cyclic_stabilizer}
that $Q$ is elementary abelian, since $Q$ fixes a point of $\mathbb{P}^1(F)$ and is a $p$-group.

3. Let $s$ be a non-identity element of $Q$. 
By part 2, $s$ has order $p$.
We claim that any element of $G$ that commutes with $s$ also has order $p$.
Since $s$ has order $p$, a pre-image $S$ of $s$ in $GL(2,E)$ has both eigenvalues equal to 1
(see Section \ref{la}).
Therefore $S$ is similar to the matrix  ${1\ 1\choose 0\ 1}$ over $E$.
This shows that a conjugate of $s$ is $x\mapsto x+1$.
Therefore, we may assume for the purposes of this proof that $s(x)=x+1$.
Note that $s$ fixes $\infty$, and this is its unique fixed point on $\mathbb{P}^1(F)$.

Let $t\in G$ commute with $s$. 
Then $t(\infty)=t s (\infty)=s t (\infty)$ so $s$ fixes $t(\infty)$.
 By uniqueness of the fixed point of $s$ we have $t(\infty)=\infty$.
 Thus we can write $t(x)=cx+d$, where $c\neq 0$ and $d$ are elements of $F$. We now calculate that
\[
st(x)=cx+d+1,\quad ts(x)=cx+c+d.
\]
Since $st=ts$  we have $c=1$ and so $t(x)=x+d$. This shows that $t$ has order $p$, as claimed.

We finally show that $t$ is in $Q$. For we have just shown that the centralizer of $s$ is a $p$-group, 
and it contains $Q$, as $Q$ is abelian. Thus $Q$ is the centralizer of $s$. 

4. We may assume that $Q$ is not the unique Sylow $p$-subgroup of $G$, for otherwise there is nothing to prove. Let $Q_1$ be a different Sylow $p$-subgroup and let
$g$ be an element of $Q\cap Q_1$. Then both $Q$ and $Q_1$ centralize $g$ and hence $g$ must be the identity, by what we proved in the previous paragraph.

5. To count the number of elements of order $p$, we recall that any element of order $p$ is in some Sylow $p$-subgroup, and this Sylow subgroup is unique, by our argument above. 
Furthermore, all Sylow $p$-subgroups are conjugate in $G$ and their number is $|G:N_G(Q)|$. 
Given this data, our statement about the number of elements of order $p$ is immediate. 

Sylow's theorem states that the number of Sylow $p$-subgroups is congruent to 1 modulo $p$, and our statement about the number of elements of order $p$ follows. (We remark that this congruence about $M$ is easily proved from first principles and does not require our special hypotheses.)

6. We have shown that $Q$ fixes $\infty$, and observed that $\infty$ is the unique fixed point. 
If $h$ is in $N_G(Q)$, it is clear that any element  of $Q$ fixes $h(\infty)$. 
Thus, by uniqueness, $N_G(Q)$ also fixes $\infty$. Let $H$ be the stabilizer in $G$ of $\infty$. 
By Lemma \ref{cyclic_stabilizer}, $H$ has a normal Sylow $p$-subgroup, which must be $Q$, since $Q$ is a Sylow $p$-subgroup of $G$ contained in $H$.
Thus $H$ is contained in $N_G(Q)$ and we conclude that $H=N_G(Q)$.

 \end{proof}

 We have shown that there are at most three non-regular orbits. 
 Next we shall show that there are at most two non-regular orbits when $p$ divides $|G|$.

\begin{lemma} \label{two_non_regular_orbits}
Let $G$ be a finite subgroup of $PGL(2,\overline{\mathbb{F}}_q)$,  where $p$ divides $|G|$. Then in its action on  $\mathbb{P}^1(\overline{\mathbb{F}}_q)$, $G$ has at most two orbits that
are not regular. $G$ has exactly one non-regular orbit if and only if $G$ is a $p$-group.
\end{lemma}

\begin{proof}
We employ the notation previously used. Suppose by way of contradiction that $G$ has exactly three 
non-regular orbits, $\Omega_i$, $1\leq i\leq 3$, with one-point stabilizers $G_i$, where $|G_1|\leq |G_2|\leq |G_3|$.
Let $M$ be the number of elements of order $p$ in $G$.

We have proved that $|G_1|=2$ and $2\leq |G_2|\leq 3$. Furthermore, if $Q$ is a Sylow
$p$-subgroup of $G$, Lemma \ref{Sylow} shows that we may take
$N_G(Q)$ to be one of the $G_i$.

We consider the possibility that 
$|G_2|=2$ also. We claim that $N_G(Q)$ cannot be $G_1$ or $G_2$. For if this is the case, $p=2$, and $G_1$ and $G_2$ are both Sylow 2-subgroups of $G$, of order 2. But then $G_1$ and $G_2$ are conjugate in $G$, and hence an element of order 2 in $G$ fixes points in both $\Omega_1$ 
and $\Omega_2$, contrary to the fact that an element of order $p=2$ fixes only one point.

Thus we can assume that $G_3=N_G(Q)$. We have additionally shown that 
\[
(M+2)|G_3|=|G|.
\]
We now have 
\[
(M+2)=|G:N_G(Q)|, \quad M=|G:N_G(Q)|(|Q|-1)
\]
and these two equations are clearly incompatible. 

We are left with the possibilities that $|G_2|=3$ and $3\leq |G_3|\leq 5$. Corresponding to the three values for $|G_3|$, there are accompanying equalities
\begin{equation}\label{eq1}
6(M+2)=|G|, \quad 12(M+2)=|G|, \quad 30(M+2)=|G|.
\end{equation}
It follows that $p$ is one of 2, 3 or 5.

Suppose first that $p=5$. Then we can only have $|G_3|=5$ and $G_3=N_G(Q)$. Our formula for
$M$ yields that $M=4|G|/5$ and this is incompatible with $30(M+2)=|G|$. Next, consider the case that $p=3$. Since no $|G_i|$ is divisible by 9, it follows that $|Q|=3$. Moreover, if
$N_G(Q)$ is not equal to $Q$, its order is at least 6, and no $G_i$ has order greater than 5. We deduce that $N_G(Q)=Q$ here and $M=2|G|/3$. This is incompatible with all of the equalities \eqref{eq1},
one of which must hold true. This excludes the case that $p=3$.

The possibility that $p=2$ remains to be excluded. Now $M+2$ is odd, and thus a Sylow
2-subgroup $Q$ of $G$ has order at most 4. Suppose that $|Q|=4$. Then we must have $N_G(Q)=Q$ and since both 6 and 30 are not divisible by 4, $12(M+2)=|G|$ must hold. Furthermore,
$M=3|G|/4$. This clearly leads to a contradiction. Finally, suppose that $|Q|=2$. We again must have $N_G(Q)=Q$ and then $M=|G|/2$. This also leads immediately to a contradiction, 
and we deduce that there is no example of a group with exactly three non-regular orbits and order divisible by $p$.

We now consider the case that $G$ has exactly one non-regular orbit, $\Omega$, say. Our previous arguments imply that
\[
|G|=|\Omega|+M+2(|G|-M-1)
\]
and thus 
\[
M=|\Omega|+|G|-2.
\]
Now we clearly have the trivial inequality $M\leq |G|-1$, with equality only if $G$ is a $p$-group, and it follows that $G$ is indeed a $p$-group and $|\Omega|=1$ (so that $\Omega$ consists of a single $G$-fixed point).
\end{proof}

{\bf Remark.}
An interesting example to consider is that of $PGL(2,q)$, of order $q^3-q$. 
By Lemma  \ref{two_non_regular_orbits} there are exactly two non-regular orbits.
This group acts
triply transitively on $\mathbb{P}^1(\mathbb{F}_q)$, and a point stabilizer is
the normalizer of a Sylow $p$-subgroup. 
The group also acts transitively on 
$\mathbb{P}^1(\mathbb{F}_{q^2})\setminus \mathbb{P}^1(\mathbb{F}_q)$, with point stabilizer 
being a cyclic group of order $q+1$. 
This accounts for the two non-regular orbits of the group on $\mathbb{P}^1(F)$. 
This is well known, see Theorem 11.92 of \cite{HKT} for example.
We note especially that $\mathbb{P}^1(\mathbb{F}_{q^3})\setminus \mathbb{P}^1(\mathbb{F}_q)$ is a single regular orbit.
These facts will be used later.

 \subsection{More on Fixed Points}
 
 We return to some results about fixed points, using Lemma \ref{Sylow}.

 \begin{lemma}\label{powersfix}
 Let $F$ be a field of characteristic $p>0$.
 Suppose $s$ is an element of 
 $PGL(2,F)$ and suppose that $s^r$ is not the identity and  fixes a point $\omega$ of $\mathbb{P}^1(F)$.
 Then $s$ also fixes $\omega$. 
 \end{lemma}

 \begin{proof}
 First suppose $F$ is algebraically closed, and assume $s$ has order relatively prime to $p$. 
 Then the same is true for $s^r$.
  By Lemma \ref{two_fixed}, $s$ and $s^r$ both fix exactly two points. 
  But any point fixed by $s$ is fixed by $s^r$, so
  the fixed points of $s$ and $s^r$ are identical.

If $p$ divides the order of $s$, then the order of $s$ must be exactly $p$.
To see this, suppose not and suppose that the order of $s$ is $kp$ where $k>1$ is an integer.
Then $s^k$ has order $p$ and commutes with $s$. 
However the proof of Lemma \ref{Sylow} part 3
shows that if any $t\in PGL(2,F)$ has order $p$, then any non-identity element that commutes with $t$ 
must have order $p$ as well.
So we have a contradiction and  the order of $s$ must be exactly $p$.

Assume now that $s$ has order $p$ and that $s^r$ is not the identity. 
Then $s^r$ also has order $p$, and $r$ must be relatively prime to $p$.
There exist integers $a,b$ such that $ap+br=1$, so $s^{br}=s$ and $s$ is a power of $s^r$.
 Since $s$ and $s^r$ are powers of each other, they have the same fixed points.

If $F$ is not algebraically closed, a fixed point in $\mathbb{P}^1(F)$ is also a fixed point in 
$\mathbb{P}^1(\overline{F})$, 
and we apply the argument above.
 \end{proof}

 Next we record a result for specific subgroups that do not have fixed points over $\F_q$.
 We will use this later.

 \begin{lemma}\label{nofixed}
 Let $G$ be a cyclic subgroup of $PGL(2,q)$ of order $r$, where $r>2$ and $r$ divides $q+1$.
 Then every non-identity element of $G$ has no fixed 
 points in the action on the projective line $\mathbb{P}^1 (\F_q)$.
\end{lemma}
 
 \begin{proof}
 First assume that $r$ has an odd prime divisor, call it $\ell$.
Let $h$ be an element of $G$ of order $\ell$.
Let $\chi(h)$ be the number of fixed points of $h$, in its action on $\mathbb{P}^1 (\F_q)$.
Since $\ell$ is a prime, every orbit has size 1 or $\ell$, and so $\chi(h) \equiv q+1 \pmod{ \ell}$ .
Since $\ell$ divides $q+1$ this becomes $\chi(h) \equiv 0 \pmod{ \ell}$.
However $\chi(h) \le 2$ and $\ell >2$, so $\chi(h)$ = 0. 
 This proves that an element of odd prime order in $G$ cannot have a fixed point.
 
 Suppose $h\in G$ has order $n$ where $n$  has an odd prime divisor $\ell$.
 Then $h^{n/\ell}$ has order $\ell$.
 If $h$ has a fixed point, then $h^{n/\ell}$ has  the same fixed point, which is impossible by 
 the previous paragraph.
 This proves that an element in $G$ whose order has an odd prime divisor cannot have a fixed point.
 
 If $r$ is odd then the proof is complete. 
 In particular, if $q$ is even then the proof is complete. 
 So we assume now that $q$ is odd and that $r$ is even.
 
 It remains to show that elements in $G$ whose order is a power of 2 do not have fixed points.
 An appropriate power of such an element will give an element of order 2, and this power
 will have the same fixed points. 
 To complete the proof  it therefore suffices to show that elements of order 2 do not have fixed points.
 The argument breaks into cases depending on $q$ modulo 4.
 
 By Lemma \ref{two_fixed} an element of order 2 has either no fixed points or two fixed points.
 
 Suppose $q\equiv 3 \pmod4$.
 This includes the case that $r$ is a power of 2, because in that case  $4|r$ and $r|(q+1)$ so $4|(q+1)$.
 If an element of order 2 has two fixed points, then its disjoint cycle decomposition
is a product of $(q-1)/2$ transpositions. So its sign is $(-1)^{(q-1)/2}$ which is $-1$.
However, when $q\equiv 3 \pmod4$ an element of order 2 is a square and so has sign $+1$.
This contradiction implies that an element of order 2 has no fixed points. 
 
  Suppose $q\equiv 1 \pmod4$. 
  Since $r$ is not a power of 2, $G$ has an element of odd order.
 Let $\tau$ be an element of order 2 in $G$, which exists by assumption,
 and let $\sigma$ be an element of odd  order in $G$.
 Suppose for the sake of contradiction that $\tau$ has two fixed points, $a$ and $b$.
 Since $G$ is abelian, $\tau \sigma (a)=\sigma \tau (a)=\sigma(a)$ and the same for $b$.
 So $\sigma(a)$ and $\sigma(b)$ are also fixed points for $\tau$.
 Because $\tau$ has exactly two fixed points, and $\sigma$ has no fixed points,
 we must have $\sigma(a)=b$ and $\sigma(b)=a$.
 But then $\sigma^2(a)=a$ which is a contradiction by Lemma \ref{powersfix}
 and the fact that $\sigma$ has odd order.
 Therefore $\tau$ has no fixed points, and the proof is complete.
 \end{proof}
 
 Finally for this section, we present an application giving more detailed information
 about the action on $\mathbb{P}^1 (\overline{\F_q})$.

  \begin{thm}\label{allregtwofixed}
  Let $G$ be a cyclic subgroup of $PGL(2,q)$ of order $r$, where $r>2$ and $r$ divides $q+1$.  
 Then $G$ has exactly two fixed points in its action on $\mathbb{P}^1 (\overline{\F_q})$, 
 and all other orbits are regular.
 Furthermore, $G$ acts regularly on $\mathbb{P}^1(\F_q)$.
 \end{thm}
 
\begin{proof}
In a non-regular orbit $\Omega$, there is an element $\omega$ such that $\omega$ is a fixed
point for some nonidentity element of $G$.
 By Lemma \ref{powersfix}, $\omega$ is a fixed point for a generator of $G$,
 and therefore is a fixed point for the whole group $G$.
 So the non-regular orbit consists of just one element, i.e., $\Omega = \{ \omega\}$.

By  Lemma \ref{two_fixed}, $G$ has exactly two fixed points.
Therefore, $G$ has exactly two non-regular orbits, each of size 1, and all other
orbits are regular.

The fixed points are not in 
 $\mathbb{P}^1(\F_q)$ by Lemma \ref{nofixed}.
 Therefore $G$ acts regularly on $\mathbb{P}^1(\F_q)$.
\end{proof}

 \subsection{Commuting Actions}\label{sect_comm_action}

 We want to consider next the following situation. We have a finite subgroup $G$ of $PGL(2,F)$. We assume that $G$ acts by fractional transformations on the roots (in some extension field of $F$) of a polynomial $f$ in $F[x]$. We will also assume that $G$ acts transitively and regularly on the roots. Furthermore, we will assume that $f$ is irreducible over $F$ so
 the Galois group acts transitively on the roots. 
 The key point is that the actions of $G$ and the Galois group on the roots commute. 
 We abstract this situation in the following two lemmas, each of which is useful.
 These lemmas will be applied later.
 
\begin{lemma}\label{comm_action1}
Let $\Omega$ be a finite set and let $G$ and $H$ be subgroups of the group of all permutations of $\Omega$. Suppose that $G$ acts transitively and regularly on $\Omega$, and $H$ acts transitively. Suppose also that
$gh(\alpha)=hg(\alpha)$
for all $g$ in $G$, $h$ in $H$ and $\alpha$ in $\Omega$. Then $H$ acts regularly on $\Omega$
and is isomorphic to $G$.
\end{lemma}

\begin{proof}
Let $\alpha$ be an element of $\Omega$ and let $H_\alpha$ be the stabilizer subgroup of $\alpha$ in $H$. Let $\beta$ be any other element in $\Omega$. Since $G$ acts transitively on $\Omega$, there is some $g\in G$ with $\beta=g(\alpha)$. Let $h$ be an element of $H_\alpha$. We have then
\[
h(\beta)=h(g(\alpha))=g(h(\alpha))=g(\alpha)=\beta.
\]
Thus $h$ fixes $\beta$ and since $\beta$ is an arbitrary element of $\Omega$, $h$ acts trivially on $\Omega$ and hence is the identity. It follows that $H_\alpha=1$ and we deduce that $H$ acts regularly on $\Omega$. 

Given $\alpha$ as above, let $g$ be any element of $G$. We define $g^*$ in $H$ to be the unique element in $H$ with $g^*(\alpha)=g(\alpha)$. We claim that $g\to g^*$ is an injective map from $G$ onto $H$. For suppose that $g^*=g_1^*$ for $g_1$ in $G$. Then we have 
$g^*(\alpha)=g_1^*(\alpha)$ and hence $g(\alpha)=g_1(\alpha)$.
This implies $g=g_1$, since $G$ acts regularly, and proves our claim.

We now wish to show that $g\to g^*$ is an anti-homomorphism, in other words, that
\[
(g_1g)^*=g^*g_1^*
\]
for all $g$ and $g_1$ in $G$. We have
\[
g^*g_1^*(\alpha)=g^*g_1(\alpha)=g_1g^*(\alpha)=g_1g(\alpha)=(g_1g)^*(\alpha),
\]
where we have used the fact that the actions of $G$ and $H$ commute. This proves that
\[
(g_1g)^*=g^*g_1^*,
\]
as required. 

To conclude, we define a map $\theta:G\to H$ by 
\[
\theta(g)=(g^*)^{-1}.
\]
The map $\theta$ is a homomorphism, since
\[
\theta(g_1g)=((g_1g)^*)^{-1}=(g^*g_1^*)^{-1}=(g_1^*)^{-1}(g^*)^{-1}=\theta(g_1)\theta(g).
\]
We thus have a monomorphism from $G$ to $H$, and since the groups have the same order $|\Omega|$, for they both act regularly on $\Omega$, $G$ and $H$ are isomorphic.
\end{proof}

Here is a generalization.
 
\begin{lemma}\label{comm_action2}
 Let $G$ and $H$ be subgroups of the group of permutations of a finite set $\Omega$. Suppose that $G$ acts transitively and regularly on $\Omega$. Suppose also that the actions of $G$ and $H$ commute. Then $G$ transitively permutes the $H$-orbits on $\Omega$, and if $T$ is the stabilizer in $G$ of a given $H$-orbit, $T$ is isomorphic to $H$.
\end{lemma}

\begin{proof}
Let $\Omega'$ be an $H$-orbit on $\Omega$. Then, for all $g$ in $G$, $g\Omega'$ is also
an $H$-orbit, since the $G$ and $H$ actions commute. Thus $G$ permutes the $H$-orbits, and indeed they form a system of imprimitivity for $G$. As is well known, and easy to prove, $G$ permutes these orbits transitively, since it acts transitively on $\Omega$.

Let $T$ be the stabilizer in $G$ of $\Omega'$. It is again a general principle of permutation group theory that $T$ acts transitively on $\Omega'$, again because $G$ acts transitively. The action of $T$ is regular, because $G$, and hence all its subgroups, acts regularly on $\Omega$. 

We now show that $H$ acts faithfully on $\Omega'$, and hence may be considered to be a subgroup of the group of permutations of $\Omega'$. For let $K$ be the kernel of the action of $H$ on $\Omega'$ and let $\omega$ be any point of $\Omega'$. Let $g$ be any element of $G$ and $h$ any element of $K$. Then since $g$ and $h$ commute, we have
\[
h(g \omega)=g(h \omega)=g(\omega)
\]
and we see that $h$ also fixes $g \omega$. As any element of $\Omega$ is expressible as
$g \omega$ for some $\omega$ in $\Omega'$, it follows that $h$ acts trivially on $\Omega$ and hence is the identity.

The conclusion of the lemma now follows from Lemma \ref{comm_action1}.
\end{proof}

 \section{Orbit Polynomials}\label{orbitp}
 
 Let $F$ be a field,  let $x$ be an indeterminate over $F$, and let $F(x)$ be the 
 field of rational functions in $x$.
 Let $G$ be a finite group of automorphisms of $F(x)$,
 i.e., a finite subgroup of $PGL(2,F)$.
 We denote the image of $x$ under $s\in G$ by $s(x)$.

 Let $T$ be another indeterminate.
 We define the \emph{orbit polynomial} of $G$ to be
 \[
 P_G(T)=\prod_{s\in G} \bigl(T-s(x) \bigr),
 \] 
 which is an element of the polynomial ring $F(x)[T]$.
 Note that $P_G(T)$ also depends on the field $F$, but we suppress this from the notation.
 This polynomial can be found in many previous works, such as \cite{GS} for example.

 Both the roots and the coefficients of $P_G(T)$ are elements of $F(x)$.
 It is clear that the coefficients of $P_G(T)$ are fixed by every element of $G$,
 and thus the coefficients lie in $F(x)^G$.
 Some coefficients in $P_G(T)$ may be constant, however at least one coefficient
 must be non-constant for otherwise $x$ would be algebraic over $F$.
 The orbit polynomial $P_G(T)$ has $x$ as a root, and is the minimal polynomial
 of $x$ over the field $F(x)^G$.
 The orbit polynomial appears in the proof of L\"uroth's Theorem in \cite{vdw}.
 All the nonconstant coefficients in $P_G(T)$ are generators for $F(x)^G$, as shown there.
 

  \subsection{Orbit Polynomial Basics}
 
 We make some elementary observations about the orbit polynomial.
The first two lemmas may be found on the 
 arxiv in a paper by Bluher \cite{B}.

 \begin{lemma}\label{orbitpoly1}
 Let $F$ be a field.
 Let $G$ be a subgroup of $PGL(2,F)$.
Let $\Phi(x)=f(x)/g(x)$ be a generator for $F(x)^G$, the field of $G$-invariant rational functions.
Let $\alpha \in \overline{F}$.

If $\alpha$ belongs to a regular orbit for $G$,  then 
$G$ acts regularly on the
roots of $f(T) - \Phi(\alpha) g(T)$ and
\[
f(T) - \Phi(\alpha) g(T)=\prod_{s\in G} \bigl(T-s(\alpha) \bigr).
\]
If $\alpha$ belongs to a non-regular orbit, then
\[
f(T) - \Phi(\alpha) g(T)=\prod_{s\in G/G_\alpha} \bigl(T-s(\alpha)\bigr)^{|G_\alpha|}
\]
where $G_\alpha$ is the stabilizer of $\alpha$, and $G/G_\alpha$ is a set of coset representatives.
\end{lemma}

 \begin{proof}
 Consider the polynomial $f(T) - \Phi(x) g(T)$, as an element of $F(x)[T]$. 
The coefficients are in $F(x)$ (actually in $F(x)^G$), and the degree is $|G|$.
This polynomial has $x$ as a root.
Furthermore, if $s\in G$ then $s(x)$ is also a root, because $\Phi(x)$ is invariant under $G$
which implies $\Phi(x)=\Phi(s(x))=f(s(x))/g(s(x))$.
So the roots are all the images of $x$ under $G$, and these are distinct.
Therefore this polynomial is equal to the orbit polynomial, $\prod_{s \in G} (T - s(x))$.

Now replace $x$ with $\alpha \in \overline{F}$.

On the one hand we get the polynomial $f(T) - \Phi(\alpha) g(T)\in  \overline{F}[T]$.
On the other hand, we get the orbit polynomial of $G$
 with $x$ replaced by $\alpha$, namely $\prod_{s \in G} (T - s(\alpha))$.
The images of $\alpha$ under the group elements may or may not be distinct,
depending on whether $\alpha$ comes from a regular or non-regular orbit.
If $\alpha$ belongs to a regular orbit for $G$,  then 
$G$ acts regularly on the
roots of $f(T) - \Phi(\alpha) g(T)$ and
\[
f(T) - \Phi(\alpha) g(T)=\prod_{s\in G} \bigl(T-s(\alpha) \bigr)
\]
and if $\alpha$ belongs to a non-regular orbit, then
\[
f(T) - \Phi(\alpha) g(T)=\prod_{s\in G/G_\alpha} \bigl(T-s(\alpha)\bigr)^{|G_\alpha|}
\]
because each element of the orbit has a point stabilizer which is a coset of $G_\alpha$.
\end{proof}

Let $\alpha \in \overline{F}$.
We will define $F_\alpha (T):=f(T) - \Phi(\alpha) g(T)$ for the rest of the paper.

\begin{lemma}\label{Phiorbit}
Let $G$ be a finite subgroup of $PGL(2,F)$, where $F$ is a field. Let $\Phi(x)=f(x)/g(x)$ be a generator of $F(x)^G$.
Let $\alpha$ and $\beta$ be elements of $\mathbb{P}^1 (F)$.
Then $\Phi(\alpha)=\Phi(\beta)$ if and only if $\alpha$ and $\beta$ are in the same $G$-orbit.
\end{lemma}

\begin{proof}
If $\alpha$ and $\beta$ are in the same $G$-orbit, then $\Phi(\alpha)=\Phi(\beta)$ because 
$\Phi(x)$ is invariant under $G$.

Conversely, suppose $\Phi(\alpha)=\Phi(\beta)$.
Then the polynomials $F_\alpha(T)$ and $F_\beta(T)$ are equal, so $\beta$ is a root of $F_\alpha(T)$.
By Lemma \ref{orbitpoly1}, $\beta$ must lie in the $G$-orbit of $\alpha$.
\end{proof}

Next we make some observations about the coefficients in the orbit polynomial.

\begin{lemma}\label{gens0}
Let $G$ be a finite subgroup of $PGL(2,F)$ and  let $|G|=m$.
Let  $F(x)$ 
be the function field over $F$ in the indeterminate $x$. 
Let the orbit polynomial of $G$ be $P_G(T)=\sum_{r=0}^m \alpha_r (x) T^r$,
where $\alpha_r\in F(x)$. 
Then, for each $r$, either $\alpha_r(x)$ is constant or 
$\alpha_r(x)=B_r(x)/A(x)$ where $B_r(x)$ and $A(x)$ are relatively prime polynomials in $F[x]$,
$\deg B_r(x)=m$, and $\deg A(x)<m$.
Furthermore  $A(x)$ does not depend on $r$, and $A(x)$ splits completely over $F$.
\end{lemma}

\begin{proof}

Let the elements of $G$ be
$$\frac{f_i(x)}{g_i(x)}=\frac{a_i x+b_i}{c_ix+d_i}$$ for $1 \leq i \leq m$,
with $f_1(x)=x$, $g_1(x)=1$, and $a_i, b_i, c_i, d_i \in F$.

The coefficient of $T^r$ in the orbit polynomial is the 
 sum over all $\binom{m}{r}$ possible  products of the $f_i(x)/g_i(x)$ taken $r$ at a time.

A typical term is say
\begin{equation}\label{typical}
\frac{f_{i_1}(x)...f_{i_r}(x)}{g_{j_1}(x) ...g_{j_r}(x)}
\end{equation}
and we add all $\binom{m}{r}$ such terms together to obtain the coefficient of $T^r$.

Let $$A(x)=g_1(x) g_2(x) \ldots g_m(x)$$ 
which is a polynomial of degree at most $m-1$ in $x$ (because $g_1(x)=1$).
Then $A(x)$ serves as a common denominator when we add the $\binom{m}{r}$ rational functions
\eqref{typical} together. 

It is clear then that the sum
of all  $\binom{m}{r}$ terms \eqref{typical}  is a rational function of degree at most $m$ of the form 
$B_r(x)/A(x)$. 
As  stated earlier,
any nonconstant coefficient of the orbit polynomial
will be a generator for $F(x)^G$, and must have degree $m$.
Therefore $B_r(x)/A(x)$  is either a constant 
or else $B_r(x)$ has degree $m$.
\end{proof}

\begin{lemma}\label{gens1}
Let $G$ be a finite subgroup of $PGL(2,F)$ and  let $|G|=m$.
Let  $F(x)$ 
be the function field over $F$ in the indeterminate $x$. 
Let the orbit polynomial of $G$ be $P_G(T)=\sum_{i=0}^m \alpha_i T^i$,
where $\alpha_i\in F(x)$. 
If $\alpha_r$ and $\alpha_s$ are both nonconstant, then there exist $\lambda, \mu \in F$
such that $\alpha_r= \lambda \alpha_s + \mu$.
\end{lemma}

\begin{proof}
Suppose $B_r/A(x)$ and $B_s(x)/A(x)$ are two non-constant coefficients of $T^r$ and $T^s$
in the orbit polynomial $P_G(T)$.
By Lemma \ref{gens0}, both $B_r(x)$ and $B_s(x)$  have leading term $x^m$ with non-zero coefficient.
Multiply $B_s(x)$ by  $\lambda \in F$ to make $B_r(x)$ and
$B_s(x)$ have the same  coefficient of $x^m$.
Then $B_r(x)/A(x)-\lambda B_s(x)/A(x)$ is a rational function of degree less than $m$ that is $G$-invariant, and so it must be a constant.
So there exists $\mu \in F$ such that 
$$\frac{B_r}{A(x)}=\lambda \frac{B_s(x)}{A(x)}+\mu.$$
\end{proof}

\subsection{Linear 1-Parameter Family}
 
 Let $t$ be an element of an extension field of $F$, which
 may be algebraic or transcendental over $F$.
 We define a \emph{linear $1$-parameter family} of polynomials in $F(t)[T]$ to be a collection of polynomials
 of the form $\sum_{i=0}^n \alpha_i T^i$ where each coefficient has the form $\alpha_i=a_it+b_i$
 for some $a_i, b_i\in F$.
 
  \begin{cor}\label{gens2}
  The orbit polynomial belongs to a linear 1-parameter family of
 polynomials. 
  \end{cor}
  
  \begin{proof}
 If we let $t\in F(x)$ be one non-constant coefficient in the orbit polynomial,
 then Lemma \ref{gens1} says that
 all other  coefficients have the form $\lambda t + \mu$ where $\lambda, \mu \in F$.
 \end{proof}
 
 In a sense, the orbit polynomial itself belongs to a family of size 1. 
 The important consequence will be proved soon, that various different specializations of the orbit
 polynomial will give rise to irreducible polynomials that all belong to a linear 1-parameter family.

 \subsection{Riemann-Hurwitz formula}\label{rhf}
 
 Recall that  $G$ is a finite subgroup of $PGL(2,F)$, where $F$ is a field. 
 Let $\Phi(x)=f(x)/g(x)$ be a generator of $F(x)^G$, the fixed field of $G$ acting on the function field $F(x)$.
 
 Note that $\Phi(x)$ is surjective when $F$ is algebraically closed, because
 the polynomial $f(x)-\lambda g(x)$ will always have roots in $F$ so the equation
 $\Phi(x)=\lambda$ has solutions.
 In fact the equation $\Phi(x)=\lambda$ has $|G|$ distinct solutions for all but a finite number
 of values of $\lambda$. 
 
 One can view the setup here in terms of a group $G$ acting on a variety $X$.
 View $X=\mathbb{P}^1 (\overline{F})$ as a genus 0 curve
 and the quotient curve $X/G$ is also of genus 0.
 The separable covering map $\Phi : X \longrightarrow X/G$ is given 
 explicitly by $\alpha \mapsto \Phi(\alpha)$.
 The  values of $\lambda$ where $\Phi(x)=\lambda$ has fewer than $|G|$ distinct solutions
 are the branch points of $\Phi$, and the solutions $x$ are the ramification points of $\Phi$. 
 The Riemann-Hurwitz formula implies
 \[
 2|G|-2\ge \sum_{P \in \mathbb{P}^1 (\overline{F})} (e_P-1)
 \]
 where $e_P$ is the ramification index at $P$, which is the order of the point stabilizer at $P$.
 Equality holds if and only if $F$ has characteristic 0
 or all the $e_P$ are not divisible by the characteristic of $F$ when $char(F)>0$.
 
 {\bf Example:} Let $F$ be an algebraically closed field and let $G$ be a subgroup
 of $PGL(2,F)$ isomorphic to $A_5$.
  If the characteristic of $F$ is not 2, 3, or 5,
 then by arguments similar to Lemma  \ref{three_non_regular_orbits}
 it can be shown that the non-regular orbits have sizes 12, 20, and 30.
  Then the point stabilizers at those points
 have orders 5,3,2 respectively, and 
 \[
  \sum_{P \in \mathbb{P}^1 (\overline{F})} (e_P-1)=12\cdot 4 + 20\cdot 2 + 30\cdot 1 =118
 \]
 which is equal to $2|G|-2$.

 The more general Riemann-Hurwitz formula (for genus 0) in characteristic $p$, which allows the possibility that
  ramification indices are divisible by $p$, states
 \[
 2|G|-2= \sum_{P \in \mathbb{P}^1 (\overline{F})} \delta_P
 \]
 where $\delta_P$ is the different at $P$.
 If $e_P$ is not divisible by $p$ (tame ramification) then $\delta_P=e_P-1$, and 
 if $e_P$ is divisible by $p$ (wild ramification)  then  
 $\delta_P=e_P+q_P-2$ where $q_P$ is the $p$-part of $e_P$.
 
 {\bf Example:} Let $F$ be an algebraically closed field and let $G$ be a subgroup
 of $PGL(2,F)$ isomorphic to $A_5$.
  If the characteristic of $F$ is 2, then 
  by arguments similar to Lemma  \ref{three_non_regular_orbits}
  it can be shown that there are two non-regular orbits of sizes 5 and 12.
  At a point $P$ in the orbit of size 12, the ramification index is $e_P=5$ so this ramification is tame,
  and $\delta_P=4$.
  At a point $P$ in the orbit of size 5, the ramification index is $e_P=12$ so this ramification is wild,
  and $\delta_P=e_P+q_P-2=12+4-2=14$. Then, to verify the Riemann-Hurwitz formula,
   \[
  \sum_{P \in \mathbb{P}^1 (\overline{F})} \delta_P=12\cdot 4 + 5\cdot 14 =118=2|G|-2.
 \]
 If one did not know $\delta_P$ at wild ramification points, it could be found using this
 formula and the orbit sizes.

 \subsection{Rest of Paper}
 
  For the rest of this paper, we will assume $F=\F_q$.
 We will let $G$ be a subgroup of $PGL(2,q)$,
  and  $\Phi(x)=f(x)/g(x)$ will be a generator for  $\F_q (x)^G$.

\section{Polynomials of the form $cx^{q+1}+dx^q-ax-b$}\label{xqplus1}

We recall that $F_\alpha (T)=f(T) - \Phi(\alpha) g(T)\in  \overline{\F_q} [T]$,  for any $\alpha \in \overline{\F_q}$,
and observe that $\alpha$ will be a root of  $F_\alpha(T)$.
We will be mostly interested in the case that $\Phi(\alpha)$ is in $\F_q$ and $\alpha \notin \F_{q}$.

We will first interpret some facts about a generator for the invariant functions 
of a subgroup of  $PGL(2,q)$ in terms of polynomials
of the form $cx^{q+1}+dx^q-ax-b$.

\begin{lemma}\label{main3a}
Let $G$ be  a subgroup of $PGL(2,q)$.
 Let $\Phi(x)$ be a generator for  $\F_q (x)^G$.
  Let $\alpha \in \overline{\F_q}$, and assume  $\alpha \notin \F_{q}$.
Then
$\Phi(\alpha)$ is in $\F_q$ if and only if there exists $s(x)=(ax+b)/(cx+d)\in G$
such that $\alpha$ is a root of $cx^{q+1}+dx^q-ax-b$, or equivalently, $s(\alpha)=\alpha^q$.
\end{lemma}

\begin{proof}
 Suppose  that $\Phi(\alpha)$ is in $\F_q$. 
Then $\Phi(\alpha)=\Phi(\alpha)^q=\Phi(\alpha^q)$,
and so $\alpha$ and $\alpha^q$ are in the same $G$-orbit by Lemma \ref{Phiorbit}. 
Therefore there exists  $s\in  G$ with $s(\alpha)=\alpha^q$. 
If $s(x)=(ax+b)/(cx+d)$, we get that $\alpha$ is a root of
$cx^{q+1}+dx^q-ax-b$.

Conversely, 
let $s(x)=(ax+b)/(cx+d)$ be in $G$, and let
  $\alpha \in \overline{\F_q}$ be a root of $cx^{q+1}+dx^q-ax-b$.
Then  $\alpha^q=s(\alpha)$.
Also  $\Phi(\alpha)=\Phi(s(\alpha))$ because $\Phi$ is $G$-invariant, so
$\Phi(\alpha)=\Phi(s(\alpha))=\Phi(\alpha^q)=\Phi(\alpha)^q$
and so $\Phi(\alpha)$  is in $\F_q$. 
\end{proof}

At this point we wish to note one consequence of this Lemma for the group $PGL(2,q)$.

\begin{cor}\label{fq23}
Let $G=PGL(2,q)$ and let $\Phi(x)$ be a generator for  $\F_q (x)^G$.
If $\alpha \in \F_{q^2}\setminus \F_q$ or $\alpha \in \F_{q^3}\setminus \F_q$
then $\Phi(\alpha)\in \F_q$.
\end{cor}

\begin{proof}
By the remarks at the end of Section \ref{nro},
the elements of $\F_{q^2}\setminus \F_q$ form one $G$-orbit.
If $\alpha \in \F_{q^2}\setminus \F_q$ this implies that there exists  $s\in  G$ with $s(\alpha)=\alpha^q$. 
The result follows from Lemma \ref{main3a}.

A similar argument holds if $\alpha \in \F_{q^3}\setminus \F_q$.
\end{proof}

The previous Lemma gives a characterization of the $\alpha$ with the property that
there exists an element of $G$ taking $\alpha$ to $\alpha^q$.
Here are some important consequences of these conditions.
This theorem is at the heart of the results in this paper.

\begin{thm}\label{main3b}
Let $G$ be  a subgroup of $PGL(2,q)$.
 Let $\Phi(x)$ be a generator for  $\F_q (x)^G$.
 Let $\alpha \in \overline{\F_q}$, with  $\alpha \notin \F_{q}$ and $\Phi(\alpha)\in \F_q$.
 Let  $s(x)=(ax+b)/(cx+d)\in G$ have the property $s(\alpha)=\alpha^q$
 (which exists by Lemma \ref{main3a}). Then
 \begin{enumerate}
\item the degree of the minimal polynomial of $\alpha$
is equal to the order of $s$ in $G$.
\item the minimal polynomial of $\alpha$ is  equal to the orbit polynomial of
the cyclic subgroup generated by $s$  specialized with $x$ replaced by $\alpha$.
\end{enumerate}
\end{thm}

\begin{proof}
1. 
Let $\sigma$ be the Frobenius automorphism on $ \overline{\F_q}$, $\sigma(x)=x^q$.
 From Galois theory we know that $\sigma$ will permute the roots of any polynomial with coefficients in $\F_q$.
 In particular, the polynomial
 $cT^{q+1}+dT^q-aT-b$ has coefficients in $\F_q$ so $\sigma$ will permute the roots.
Since $\alpha$ is a root,  $\sigma^i(\alpha)=\alpha^{q^i}$
is a root for any $i$.

We are given  $s(\alpha)=\sigma(\alpha)$. 
Next observe that
  \[
 \sigma s(\alpha)=\sigma((a\alpha+b)/ (c\alpha+d))=(a\sigma(\alpha)+b)/(c\sigma(\alpha)+d)=
 s\sigma(\alpha).
 \]
 Therefore
$s^2(\alpha)=s\sigma(\alpha)=\sigma s(\alpha)=\sigma^2 (\alpha)$ 
and inductively we obtain
$s^i(\alpha)=\sigma^i(\alpha)$ for all $i\ge 1$.

Let $\alpha$ be in $\F_{q^r}$ and no smaller field, 
which implies that the minimal polynomial of  $\alpha$ is
$ \prod_{i=0}^{r-1} \bigl(T-\sigma^i(\alpha) \bigr)$.
Assume $r>2$ for the moment.
Then $s^r (\alpha)=\alpha^{q^r}=\alpha$, and $s^i (\alpha)\not= \alpha$ for $i<r$.
Since $s$ and $\sigma$ commute we have, for any $i\ge 1$,
$s^r(\sigma^i (\alpha))=\sigma^i (s^r (\alpha))=\sigma^i (\alpha)$
and thus $s^r$ fixes all roots of the minimal polynomial of $\alpha$.
So $s^r$ has more than two fixed points,
which is impossible for a non-identity element by Lemma  \ref{two_fixed}
unless $s^r=1$.
This proves that $s$ has order $r$ in $G$.

Assume now that $\alpha$ is in $\F_{q^2}$ but not $\F_q$. 
Then $\alpha^{q^2}=\alpha=s^2(\alpha)$. 
So $s^2$ fixes $\alpha$ and hence by Lemma \ref{powersfix}, unless $s^2=1$, $s$ also fixes $\alpha$.
But if $s$ fixes $\alpha$, we get $\alpha^q=\alpha$, and so $\alpha$ is in $\F_q$ contrary to hypothesis. 
Therefore $s$ has order 2, which is equal to the degree of
the minimal polynomial of $\alpha$.
 
 In all cases, we have shown that the order of $s$ is equal to the degree of the minimal polynomial of $\alpha$.

2. 
The minimal polynomial of $\alpha$ is $ \prod_{i=0}^{r-1} \bigl(T-\sigma^i(\alpha) \bigr)$.
This is equal to $\prod_{i=0}^{r-1} \bigl(T-s^i(\alpha) \bigr)$ by the proof of part 1.
By definition, this is the orbit polynomial of the cyclic subgroup generated by $s(x)$,
with $x$ replaced by $\alpha$.
\end{proof}

 \subsection{Subgroups of order dividing $q+1$}

We now investigate the consequences of the previous section for cyclic subgroups of order dividing $q+1$.
Part 1 has already been proved in \cite{ST}.

 \begin{cor}\label{main1}
 Let $G$ be a cyclic subgroup 
 of $PGL(2,q)$ of order $r>2$ generated by $s(x)=\frac{ax+b}{cx+d}$.
 Assume  $r$ divides $q+1$.
 Let $p_s(T):=cT^{q+1}+dT^q-aT-b \in \F_q [T]$.
 Then
\begin{enumerate}
\item The polynomial $p_s(T)$ factors into irreducible
factors of degree $r$. 
\item Each irreducible factor  is the minimal polynomial of a root $\alpha\in \overline{\F_q}$ of $p_s(T)$,
and is obtained from  the orbit polynomial $P_G(T)$  of $G$
by replacing the  indeterminate $x$ with  $\alpha$.
\item The irreducible factors are all from a linear 1-parameter family of polynomials of degree $r$ in $ \F_q [T]$.
  In particular, any coefficients that are constant in $P_G(T)$ will have
 that same constant value in every irreducible factor.
 \end{enumerate}
 \end{cor}

 \begin{proof}
 1. This follows from Theorem \ref{main3b} because $s$ has order $r$ (apply Theorem \ref{main3b} part 1 to any root of $p_s(T)$).
 
 2. This follows from Theorem \ref{main3b}.
 
 3. This follows from Corollary \ref{gens2} and  part 2.
 \end{proof}

     {\bf Example:} (from the introduction)
    Suppose $q\equiv 3\pmod{4}$, so that 4 divides $q+1$.
 Consider  $ s(x)=\frac{-x-1}{x-1}$, an element of $PGL(2,q)$ 
 which has order 4.
 The orbit polynomial of the subgroup generated by $s$ is
 \[
 T^4-\frac{x^4-6x^2+1}{x^3-x} \ T^3  - 6T^2  + \frac{x^4-6x^2+1}{x^3-x}\ T +1.
 \]
 It follows from Corollary \ref{main1} that
 all irreducible factors of $T^{q+1}-T^q+T+1\in \F_q [T]$ will have this form,
 when $x$ is replaced by a root of $T^{q+1}-T^q+T+1$.

 This explains the patterns in the $q=19$  example in the introduction.
 All factors belong to the linear 1-parameter family $T^4-tT^3-6T^2+tT+1$.

We remark that  this linear 1-parameter family appears over the rational numbers
in number theory.  They define extensions of $\mathbb{Q}$ known as the simplest quartic fields,
see \cite{F} for example.

 \bigskip

{\bf Example:} 
Let $q=17$, and let $G$ be the cyclic subgroup of order 3 generated by 
 $s(x)=(14x+13)/(6x+2)$.
 The orbit polynomial of $G$ is
 \[ P_G(T)=T^3 + \frac{16x^3 + 2x + 10}{x^2 + 15x + 3}T^2 + \frac{2x^3 + 15x^2 + 8}{x^2 +
    15x + 3}T + \frac{14x^3 + 7x^2 + 9x}{x^2 + 15x + 3}
    \]
 and we take the generator $\Phi(x)=  \frac{2x^3 + 15x^2 + 8}{x^2 +
    15x + 3}$.  Call this generator $t$. Then it is readily checked that
    \[
    P_G(T)=T^3 + (8t-1)T^2+tT+(7t+4).
    \]
  The degree $q+1$ polynomial corresponding to $s(x)$ is
 $6T^{18}+2T^{17}-14T-13$, which  
by Corollary \ref{main1} factorizes into six irreducible factors of degree three.
We used MAGMA \cite{Bosma}  for the factorizations.
 All the irreducible factors belong to the same linear 1-parameter family, where $t$ is obtained
 by replacing $x$ with a root $\alpha$ of the factor.
 The (nonconstant) factors are
 \begin{align*}
 &T^3 + 15T + 7\cr
 &T^3 + 3T^2 + 9T + 16\cr
 & T^3 + 4T^2 + 7T + 2\cr
    &T^3 + 6T^2 + 3T + 8\cr
  &  T^3 + 12T^2 + 8T + 9\cr
    &T^3 + 15T^2 + 2T + 1.
    \end{align*}
  Each factor belongs to the linear 1-parameter family
     $T^3 + (8t-1)T^2+tT+(7t+4)$.

     \bigskip

 Here is a corollary of Theorem \ref{main3b} and Corollary \ref{main1}.

 \begin{cor}\label{main2}
 In $PGL(2,q)$,  $s(x)=\frac{ax+b}{cx+d}$ is an element of order exactly $q+1$
 if and only if   $p_s(T):=cT^{q+1}+dT^q-aT-b \in \F_q [T]$ is irreducible over $\F_q$. 
 
   The orbit polynomial
   of the cyclic subgroup generated by $s$ becomes $p_s(T)$ when
   $x$ is replaced by any root of $p_s(T)$.
 \end{cor}
 
 We will extend this Corollary  after some further lemmas in the next section.

 \subsection{On numerators and denominators for subgroups of order $q+1$}
 
 We present an application of the earlier lemmas about fixed points 
 that yields additional information about the numerators and denominators
 of elements of the cyclic subgroups of order dividing $q+1$.

 \begin{lemma}\label{fqdenom}
 Let $s(x)$ be an element of order $r$ in  $PGL(2,q)$,
 where $r>2$ and $r$ divides $q+1$.
For $i=1,2, \ldots ,r$  write
$s^i(x)=f_i(x)/g_i(x)$
where $f_i$ and $g_i$ are polynomials of degree at most 1.
Then no denominator polynomial
$g_i$ is a scalar multiple of $g_j$ if $i$ is different from $j$. 
The same holds for the numerators $f_i$.
\end{lemma}

\begin{proof}
Suppose the contrary. After multiplying by a scalar if necessary, we may assume that
$g_i(x)=g_j(x)=cx+d$.
Write $f_i(x)=a_ix+b_i$, $f_j(x)=a_jx+b_j$.

Then if we let $z=(b_j-b_i)/(a_i-a_j)$, we get $s^i(z)=s^j(z)$ and 
so $s^{i-j}$ has a fixed point on the projective line over $\F_q$.
(If $a_i=a_j$ and $b_i\not= b_j$ the fixed point will be $\infty$,
otherwise the fixed point is an element of $\F_q$.)
This contradicts Lemma  \ref{nofixed}.

The same argument works for the numerators.
\end{proof}

\begin{lemma}\label{fqdenom2}
 Let $s(x)$ be an element of order exactly $q+1$ in  $PGL(2,q)$.
For $i=1,2, \ldots , q+1$  write $s^i(x)=f_i(x)/g_i(x)$
where $f_i$ and $g_i$ are polynomials of degree at most 1.
Then $g(x)=\prod_{i=1}^r g_i(x)$ is a scalar multiple of $x^q-x$, 
and the same is true of the product of the $f_i$.
\end{lemma}

\begin{proof} 
By Lemma \ref{fqdenom}
there are $q+1$ linearly independent polynomials $g_i(x)$, and one of them is constant.
Taking out the leading coefficients, we must have 
$$\prod_{i=1}^{q+1} g_i(x)=c \prod_{\alpha \in \F_q} (x-\alpha)=c(x^q-x)$$
where $c\in \F_q$. The same argument works for the numerators.
\end{proof}

Remark: Lemma  \ref{fqdenom2} implies that, for a cyclic subgroup of order $q+1$,
the numerator of a generator  of the invariant function field will never have any roots in $\F_q$.

In the case of a cyclic subgroup of order exactly $q+1$, Lemma \ref{fqdenom2} and 
Theorem \ref{allregtwofixed}
say that the roots of $g(x)$ together with $\infty$ is exactly $\mathbb{P}^1(\F_q)$,
and this is one regular orbit for the subgroup. 
We now generalize this to divisors of $q+1$.

 \begin{thm}\label{fqdenom3}
 Let $s(x)$ be an element of order $r$ in  $PGL(2,q)$,
 where $r>2$ and $r$ divides $q+1$.
 Let $G$ be the cyclic subgroup generated by $s(x)$.
For $i=1,2, \ldots ,r$  write
$s^i(x)=f_i(x)/g_i(x)$
where $f_i$ and $g_i$ are polynomials of degree at most 1.
Let $g(x)=\prod_{i=1}^r g_i(x)$. Then  the roots of $g(x)$ together with $\infty$
form a regular $G$-orbit.
 \end{thm}
 
 \begin{proof}
 By Lemma  \ref{fqdenom} no denominator polynomial
$g_i$ is a scalar multiple of $g_j$ if $i$ is different from $j$. 
Recall that $f_r(x)=x$ and $g_r(x)=1$.
Therefore the other denominators $g_i(x)$ for $1\le i \le r-1$ have degree exactly 1.

 Let $f(x)=\prod_{i=1}^r f_i(x)$ and $g(x)=\prod_{i=1}^r g_i(x)$.
 Then $f(x)$ has degree $r$, and $g(x)$ has degree $r-1$ with all roots distinct and in $\F_q$.
 Also $f(x)/g(x)$ is invariant under $G$; 
 it is the constant (norm) term in the orbit polynomial of $G$.
 
 Recall from Theorem \ref{allregtwofixed} that $G$ acts regularly on $\mathbb{P}^1(\F_q)$.
 Thus the $q+1$ elements of $\mathbb{P}^1(\F_q)$ will break into $(q+1)/r$ orbits of size $r$.
 We will show that the $r-1$ roots of $g(x)$ together with $\infty$ form one orbit.
  It  suffices to show that if $\beta\in \F_q$ is not a root of $g$, then $s^i(\beta)$ is also not a root for any $i$.

 Suppose $\beta\in \F_q$ is not a root of $g(x)$.
 Then $f(\beta)/g(\beta)$ is a well defined element of $\F_q$. 
 Since $f(x)/g(x)$ is invariant under $G$,  for any $s^i$ in $G$
\[
f(s^i(\beta))/g(s^i(\beta))=f(\beta)/g(\beta)
\]
is the same element of $\F_q$. 
Thus $s^i(\beta)$ is not a root of $g(x)$ and the proof is complete.
\end{proof}
 
 Let $G$ be a cyclic subgroup of $PGL(2,q)$ of order dividing $q+1$.
Let $\Phi(x)=f(x)/g(x)$ be a generator for $\F_q(x)^G$, the field of $G$-invariant rational functions.
By Lemma \ref{Phiorbit} we know that $\Phi$ is constant on orbits of $G$.
The previous results identify precisely 
the orbit of elements $x$ such that $\Phi(x)=\infty$.

 \subsection{Subgroups of order $q+1$ and the Orbit Polynomial}
 
 We now apply the lemmas just proved to coefficients in the orbit polynomial.
 
  \begin{thm}\label{orderq1}
  Suppose $s(x)=\frac{ax+b}{cx+d}$ is an element in $PGL(2,q)$ of order exactly $q+1$.   
  Let $G$ be the cyclic subgroup generated by $s(x)$.
 Suppose the coefficient of $T^i$ in the orbit polynomial $P_G(T)$ is $f_i(x)/g(x)$,
 and is nonconstant.
  Then  there exist $\lambda_i, \mu_i \in \F_q$ such that
 $$f_i(x)=\mu_i (cx^{q+1}+dx^q-ax-b) +\lambda_i (x^q-x)$$
 \end{thm}
 
 \begin{proof}
 By Lemma \ref{fqdenom2}, $g(x)=\lambda (x^q-x)$ for some $\lambda \in \F_q$.
 
 Let $\alpha$ be a root of $p_s(T):=cT^{q+1}+dT^q-aT-b \in \F_q [T]$.
Then   $f_i(\alpha)/g(\alpha)$ is in $\F_q$ by Corollary \ref{main2}, say $f_i(\alpha)/g(\alpha)=c_i$.
This implies that $\alpha$ is a root of 
$$f_i(T)-c_i g(T)=f_i(T)-c_i \lambda (T^q-T).$$
This is a polynomial of  degree $q+1$ with coefficients in $\F_q$.
However we already know by Corollary \ref{main2} that the minimal polynomial of $\alpha$ 
over $\F_q$ is $p_s(T)$, which also has degree $q+1$. Therefore
\[
f_i(T)-c_i \lambda (T^q-T)=\mu_i f_g (T)
\]
for some $\mu_i \in \F_q$, and the result follows letting $\lambda_i=c_i \lambda$.
\end{proof}

Note in the preceding proof that $f_i(T)- \lambda_i (T^q-T)$ is irreducible over $\F_q$.
It will happen that $f_i(T)- \lambda (T^q-T)$ is irreducible over $\F_q$ for some $\lambda \in \F_q$,
and not for others.
In the next Theorem we will count the number of $\lambda$ such that $f_i(T)- \lambda (T^q-T)$ is irreducible.

\begin{thm}\label{numberofr}
 Suppose $s$ defined by $s(x)=\frac{ax+b}{cx+d}$ is an element in $PGL(2,q)$ of order exactly $q+1$.   
  Let $G$ be the cyclic subgroup generated by $s$.
 Let $f(x)/g(x)$ be a nonconstant coefficient in the orbit polynomial $P_G(T)$.
 Then $f(x)-\lambda g(x)$ is irreducible over $\F_q$ for $\phi(q+1)$ different values 
 of $\lambda \in \F_q$, where $\phi$ is the Euler phi function.
 More generally, $f(x)-\lambda g(x)$ factors into irreducible factors of degree $r$ for $\phi(r)$
 values of $\lambda \in \F_q$, for each divisor $r>1$ of $q+1$.
\end{thm}

\begin{proof}
By Theorem \ref{orderq1} we may assume that
\begin{equation}\label{eqnorderq1}
f(x)-\lambda g(x)=cx^{q+1}+(d-\lambda) x^q-(a-\lambda)x-b
\end{equation}
up to multiplication by a scalar multiple from $\F_q$.
By Corollary \ref{main2}, we are given that $cx^{q+1}+d x^q-ax-b$ is irreducible and 
we count the number of $\lambda$ such that  \eqref{eqnorderq1} is irreducible.
This is equivalent to counting the number of $\lambda$ such that 
$s_\lambda(x):=\frac{(a-\lambda)x+b}{cx+(d-\lambda)}$ has order $q+1$,
given that $s(x)=\frac{ax+b}{cx+d}$ has order $q+1$.
As $\lambda$ ranges over $\F_q$ the elements $s_\lambda$ are all distinct and not the identity.
To complete the proof, we will show that  $s_\lambda$ is in $G$  for all $\lambda$,
where $G$ is the cyclic subgroup generated by $s=s_0$. 
In other words, we will show that 
\[
G=\langle s \rangle = \bigg\{ x\mapsto \frac{(a-\lambda)x+b}{cx+(d-\lambda)} : \lambda \in \F_q \biggr\} \cup 
 \bigg\{ x\mapsto x \biggr\}.
\]
Since $G$ has $\phi(r)$ elements of order exactly $r$, this is sufficient.

Since $s$ has order $q+1$ the matrix $A={a\ b\choose c\ d}$ in $GL(2,q)$ which is a pre-image of $s$
has an irreducible characteristic polynomial and two distinct eigenvalues in $\F_{q^2}$,
call them $\alpha$ and $\alpha^q$.
In general, the order of any such matrix is a divisor of $q^2-1$ that does not divide $q-1$.
In our case, we assume that $s$ has order exactly $q+1$ so $A$ has order $q^2-1$.
Note that  $\lambda$ cannot be an eigenvalue of $A$ because $\lambda \in \F_q$.

The matrix $A_\lambda ={a-\lambda \ \ \ b\choose c\ \ \  d-\lambda }$ in $GL(2,q)$ which is a pre-image of $s_\lambda$
has eigenvalues $\alpha-\lambda$ and $\alpha^q-\lambda=(\alpha-\lambda)^q$.
These are distinct, and not in $\F_q$, and so $A_\lambda$ has order dividing $q^2-1$ and not dividing $q-1$.
Thus  $s_\lambda$ has order dividing $q+1$.

A straightforward calculation shows that $s_\lambda$ has the same fixed points as $s$
(over $\overline{\F_q}$).
Then the subgroup $H$ generated by $s$ and $s_\lambda$ has a fixed point.
The order of $H$ is relatively prime to $q$ because
both generators have order that is relatively prime to $q$.
By Lemma \ref{cyclic_stabilizer}, $H$ is cyclic.
If $s_\lambda$ is not in the cyclic subgroup $G$ generated by $s$,
then the order of $H$ would be larger than $q+1$,
which is impossible because the largest cyclic subgroups of $PGL(2,q)$ have order $q+1$.
Therefore $s_\lambda$ is in $G$.
\end{proof}

{\bf Remark.} As a check, note that $\sum_{r|(q+1), r>1} \phi (r)=q$.

{\bf Remark.}  Lemma  \ref{fqdenom2} implies that
$f(x)-\lambda g(x)$ will never have a degree 1 factor over $\F_q$.

{\bf Example:} Let $q=7$, then the cyclic subgroup generated by 
$x\mapsto (3x-1)/(x+3)$
has order 8. A nonconstant coefficient in the orbit polynomial has numerator $f(x)=x^8+1$
and denominator $g(x)=x^7-x$. 
The factorizations of $x^8+1-\lambda (x^7-x)$ for $\lambda \in \F_7$ are into
four irreducible quadratics for $\phi (2)=1$ value of $\lambda$, 
two irreducible quartics for $\phi (4)=2$ values of $\lambda$, 
and remaining irreducible for $\phi (8)=4$ values of $\lambda$.
These correspond to the group having  one element of order 2,
two elements of order 4, and four elements of order 8.

\subsection{Order dividing $q-1$}

So far we have discussed the details in the case of an element
of $PGL(2,q)$ of order $r>2$ dividing $q+1$.
The order of an element could also be $p$, or it could divide $q-1$, 
as explained in Section \ref{la} and the proof of Lemma \ref{powersfix}.
Essentially the same result holds in these cases; we omit the details and 
provide a brief sketch here.

For elements $s(x)=\frac{ax+b}{cx+d}$ of order $r>2$ dividing $q-1$, the polynomial
$cx^{q+1}+dx^{q}-ax-b$ has two roots in $\F_q$ and 
behaves in exactly the same way as for elements of order dividing $q+1$
once these two linear factors are removed.
The fact that there are two roots in $\F_q$ is essentially because these elements
have two fixed points in $\mathbb{P}^1(\F_q)$.
Corollary \ref{main1} holds when the linear factors corresponding to roots are removed.

Elements of order $p$ have just one fixed point  so the polynomial
will have one root in $\F_q$ and 
Corollary \ref{main1} holds when the linear factor corresponding to the root is removed.

Elements of order 2 are slightly trickier to deal with, partly because 2 divides both $q-1$ and $q+1$
when $q$ is odd.
Elements of order 2 fall into two conjugacy classes when $q$ is odd, see Section 7 for more details.
One conjugacy class behaves like 
elements of order dividing $q+1$, these have fixed points
defined over $\F_{q^2}$. Corollary \ref{main1} holds here, and
the polynomial $cx^{q+1}+dx^{q}-ax-b$ factors into irreducible quadratics.
The other conjugacy class have fixed points in $\mathbb{P}^1(\F_q)$ and they behave like
elements of order dividing $q-1$. 
Corollary \ref{main1} holds when the linear factors corresponding to roots are removed.
For example, when $s(x)=-x$ the corresponding polynomial is 
$x^q+x$, which factors into irreducible quadratics and one linear factor.

\subsection{Polynomials $cx^{q^k+1}+dx^{q^k}-ax-b$}\label{xqplus2}

In \cite{ST} the factors of $f_k(T)=cT^{q^k+1}+dT^{q^k}-aT-b$ were studied,
where $a,b,c,d$ are in $\F_q$.
Our results so far have been about the $k=1$ case.
For the general case,
first consider factorizing $f_k$ over $\F_{q^k}$. 
 Suppose that $s(x)=\frac{ax+b}{cx+d}$ has order $r>2$ as an element of $PGL(2,q^k)$
 where $r$ divides $q^k+1$.
Then Corollary \ref{main1} will apply with $\F_{q^k}$ as the ground field. 
We conclude that the irreducible factors over $\F_{q^k}$ must be specializations of the orbit polynomial
of the cyclic subgroup generated by $s$.

Then consider factoring $f_k$ over over $\F_q$. 
Taking one irreducible factor over $\F_{q^k}$ and multiplying it by all its Galois conjugates
gives an irreducible factor over $\F_q$. 
(The Galois conjugates must also be factors because $f_k$ has coefficients in $\F_q$.)

We must show that all irreducible factors over $\F_q$ arise in this way.
If $h$ is any irreducible factor over $\F_q$, then consider factoring $h$ over $\F_{q^k}$.
The irreducible factors of $h$ over $\F_{q^k}$ must be specializations of the 
orbit polynomial (because they are also factors of $f_k$).
Thus $h$ is the product of one of these specializations and its Galois conjugates.

\section{Polynomials of degree $|G|$}\label{degorderG}

Let $G$ be a subgroup of $PGL(2,q)$, not necessarily cyclic.
We now continue our observations about polynomials
 $F_\alpha (T)=f(T) - \Phi(\alpha) g(T)$ where 
$\Phi(x)=f(x)/g(x)$ is a generator for $\F_q(x)^G$.
The generator is not unique.
Note once again that $\Phi(x)-\lambda$ is also a generator, for any $\lambda\in \F_q$,
so $f(x)-\lambda g(x)$ is also the numerator of a generator.
While the factorization pattern of $f(x)-\lambda g(x)$  into irreducible factors over $\F_q$
may vary as $\lambda$ varies in $\F_q$, as we showed in Theorem \ref{numberofr},
it is nevertheless true that the roots of each 
$f(x)-\lambda g(x)$ form one $G$-orbit, see Lemma  \ref{orbitpoly1}.
We will now obtain more precise information about the factorization.

  \begin{thm}\label{falpha1}
  Let $G$ be a subgroup of $PGL(2,q)$.
Let $\Phi(x)=f(x)/g(x)$ be a generator for $\F_q(x)^G$, the field of $G$-invariant rational functions.
Let $\alpha \in \overline{\F_q}$ have the property that $\Phi(\alpha)$ is in $\F_q$.
Assume that $G$ acts regularly on the roots of 
$F_\alpha (T):=f(T) - \Phi(\alpha) g(T)\in \F_q[T]$.
Then 
\begin{enumerate}
\item $F_\alpha (T)$ will factor into irreducible factors of the same degree over $\F_q$.
\item The  minimal polynomial of $\alpha$  is one of the factors.
\item The degree of each factor must be the order of an element of $G$, including the possibility of
the identity element, in which case the degree will be 1.
 \end{enumerate}
  \end{thm}

\begin{proof}
We are given $\alpha \in \overline{\F_q}$ such that $\Phi(\alpha)$ is in $\F_q$, and
we are assuming that $G$ permutes the roots of $F_\alpha(T)$ regularly.
Since $\alpha$ is a root of $F_\alpha (T)$,
 the minimal polynomial of $\alpha$  over $\F_q$ divides $F_\alpha (T)$.

If $\alpha\in \F_q$ then $s(\alpha)\in \F_q$ for all $s\in G$, 
and these are all the roots of $F_\alpha(T)$,
so $F_\alpha(T)$ splits completely
over $\F_q$ into linear factors. The theorem is proved in this case.

Assume now that  $\alpha\notin \F_q$.
 The Frobenius automorphism $\sigma$ acts (nontrivially)
 on the roots of $F_\alpha (T)$, and commutes with the $G$ action.
 To see this, let $\beta$ be a root of $F_\alpha(T)$, and let $s$ be an element of $G$, with
 \[
 s(\beta)=(a\beta+b)/(c\beta+d),
 \]
 where $a$, $b$, $c$ and $d$ are elements of $\F_q$ and $ad-bc\neq 0$. Then we find that
  \[
 \sigma s(\beta)=\sigma((a\beta+b)/ (c\beta+d))=(a\sigma(\beta)+b)/(c\sigma(\beta)+d)=
 s\sigma(\beta).
 \]
 We may now apply Lemma  \ref{comm_action2} and deduce that
 the irreducible factors all have the same degree, which is  the size of a Frobenius orbit.
 
 The rest follows from Theorem \ref{main3b}.
 \end{proof}

{\bf Remark.} The $\alpha$ having $\Phi(\alpha)\in \F_q$ are the roots of the irreducible 
factors of the polynomials $f(x)-\lambda g(x)$ that we discussed in Theorem \ref{numberofr}.

{\bf Remark.}  Similar results for polynomials of the form $f(T)-\alpha g(T)$ where $\alpha$ lies in a finite extension 
of $\F_q$, and the factorization is done over the extension field, may be found in \cite{PRW}, Corollary 3.3.

Over the field $\F_{q^3}$ things are especially nice:

 \begin{cor}\label{falpha2}
 Let $G=PGL(2,q)$ and let 
  $\Phi(x)=f(x)/g(x)$ be a generator for  $\F_q (x)^G$.
Then  $\Phi(\alpha)$ is in $\F_q$ for all $\alpha \in \F_{q^3}\setminus \F_q$,
and $G$ acts regularly on the roots of $F_\alpha (T)$,
and $F_\alpha (T)$ factors over $\F_q$  into the product of all irreducible polynomials of degree $3$.
 \end{cor}
  
  \begin{proof}
  We recall that the number of irreducible polynomials of degree 3 over $\F_q$ is known to be $(q^3-q)/3$.
  The roots of all these polynomials comprise all the elements of $\F_{q^3}\setminus \F_q$.
  For any  $\alpha \in \F_{q^3}\setminus \F_q$, let $\lambda = \Phi(\alpha)$.
  Then $\lambda$ is in $\F_q$ by Corollary \ref{fq23}.
  By Theorem \ref{falpha1},
   $f(T)-\lambda g(T)$ factors over $\F_q$  into a product of  irreducible polynomials of degree $3$.
   By degrees there must be $(q^3-q)/3$ factors, so this is all irreducibles of degree 3.
  \end{proof}

 {\bf Example:}  take $q=3$ and $G=PGL(2,3)$.
  Note that $PGL(2,3)$ is isomorphic to $S_4$ and the orders of the elements are 1,2,3,4.
 
 The orbit polynomial for $G$ is
 \[ T^{24} + T^{22} + T^{20} +\Phi(x)T^{18} +
  T^{16} + T^{14}  +\Phi(x)T^{12} + T^{10} + T^8  +\Phi(x)T^6 + T^4 + T^2 + 1
    \]
    where $\Phi(x)=-f(x)/g(x)$ and 
   $f(x)= x^{24} + x^{22} + x^{20} + x^{16} + x^{14} + x^{10} +
    x^8 + x^4 + x^2 + 1$ and $g(x)=x^{18} + x^{12} + x^6$.
 
 We consider $f(x)-\lambda g(x)$ for each $\lambda \in \F_3$.
 We used MAGMA \cite{Bosma}  for the factorizations.
 
 Firstly
  $f(x)-g(x)$ factors over $\F_3$ into eight irreducible factors, all of degree 3.
 In this case $G$ has one regular orbit on the roots of $f(x)-g(x)$, and the
 Frobenius automorphism has eight orbits of size 3.
 If $\alpha$ is a root of \emph{any} of these eight cubics, 
$ \Phi(\alpha)$ takes the same value in $\F_3$, namely $1$.
This illustrates Corollary \ref{falpha2}.
 
 For the polynomial $f(x)+g(x)$  
 which factors as three quadratic factors each having multiplicity 4 ($=q+1$),
 $G$ does not act regularly on the roots. 
 The roots form a non-regular orbit of size 6  consisting of all elements
 in $\F_{q^2}\setminus \F_q$.
 
 The polynomial $f(x)$ (the case $\lambda=0$) factors into six irreducible factors of degree 4.
The roots form one regular $G$-orbit of elements from $\F_{q^4}$.

\bigskip

  As a final remark, we observe that 
  the two polynomials we have studied, 
  $f(T) -\Phi(\alpha) g(T)$ and $cT^{q+1}+dT^q-aT-b$,
  have a common factor.
  
  \begin{cor}
  Let $G$ be  a subgroup of $PGL(2,q)$.
 Let $\Phi(x)=f(x)/g(x)$ be a generator for  $\F_q (x)^G$.
   Let $\alpha \in \overline{\F_q}$ have the property that
$\Phi(\alpha)$ is in $\F_q$.  Let  $s(x)=(ax+b)/(cx+d)\in G$ have the property $s(\alpha)=\alpha^q$
(such $s$  exists by Lemma \ref{main3a}).
 Then the minimal polynomial of $\alpha$ over $\F_q$ divides both 
  $f(T) -\Phi(\alpha) g(T)\in \F_q [T]$ and $cT^{q+1}+dT^q-aT-b \in \F_q [T]$.
  \end{cor}

\section{Conjugacy Classes}

Theorem \ref{falpha1} 
can be extended into an interesting correspondence between points on the 
projective line $\mathbb{P}^1(\mathbb{F}_q)$ and conjugacy classes in $PGL(2,q)$. 
First we recall the situation with orbits.
For this section we set
$G=PGL(2,q)$ and  let $\Phi=f(x)/g(x)$ be a generator of $\mathbb{F}_q(x)^G$
with $\deg (f)=|G|$ and $\deg (g) < |G|$,
such as the generator in \eqref{pglgen}.
Consider now $G$ as acting on $\mathbb{P}^1(\overline{\mathbb{F}_q})$.
Given any element $\lambda$ of $\mathbb{F}_q$, there is an 
element $\alpha\in \overline{\mathbb{F}_q}$ such 
that $\Phi(\alpha)=\lambda$, as discussed in Section \ref{rhf}.
Lemma \ref{Phiorbit}  shows that if
$\beta\in \overline{\mathbb{F}_q}$ also satisfies $\Phi(\beta)=\lambda$, then $\alpha$ and $\beta$ are in the same $G$-orbit. 
Thus each element of $\mathbb{F}_q$ determines a different $G$-orbit, giving us $q$ $G$-orbits.
If we admit the formal possibility that $\lambda$ equals $\infty$, 
and recall that $\Phi(\alpha)=\infty$ for any $\alpha$ in the orbit  $\mathbb{P}^1(\mathbb{F}_q)$,
we have $q+1$ different orbits of $PGL(2,q)$, 
each orbit labelled by a point of $\mathbb{P}^1(\mathbb{F}_q)$.

By Lemma \ref{main3a} we know
that if $\alpha\in \overline{\mathbb{F}_q}$ has the property that
$\Phi(\alpha)\in \mathbb{F}_q$, then there is some $s\in G$ with $s(\alpha)=\alpha^q$. 
We will show that we may associate a  conjugacy class of $G$ to
this $\Phi$ value,
namely the conjugacy class  of $s$, 
provided that $\alpha$ lies in a regular $G$-orbit. 
Our main result in this section is that we obtain
different conjugacy classes of $G$ for different values of $\Phi$. Furthermore, with the exception of conjugacy classes consisting of elements of order at most 2, every conjugacy class of $G$ arises in this way. This correspondence enables us to enumerate the conjugacy classes of $G$, a well known result usually proved by other means.

\subsection{Conjugacy classes in $PGL(2,q)$}

As we have observed, $PGL(2,q)$ has exactly two non-regular orbits 
on  $\mathbb{P}^1(\overline{\mathbb{F}_q})$, by Lemma  \ref{two_non_regular_orbits}.
We intend to show how properties of these orbits can be used to deduce information
about the conjugacy classes of $PGL(2,q)$. The conjugacy classes are well studied
from a variety of points of view, but we wish to establish certain facts about them using elementary principles of group theory.

One non-regular orbit is  $\mathbb{P}^1(\mathbb{F}_q)$
of size $q+1$. The stabilizer of a point in this orbit is a subgroup $H_1$, say, of order
$q(q-1)$. The other non-regular orbit may be identified with $\mathbb{F}_{q^2}\setminus
\mathbb{F}_q$ and it has size $q^2-q$. The stabilizer of a point in the second orbit
is a subgroup $H_2$, say, of order $q+1$. Since 
by Lemma \ref{two_fixed}
each element of $PGL(2,q)$ fixes a point 
of  $\mathbb{P}^1(\overline{\mathbb{F}_q})$, we see that the subgroups $H_1$ and $H_2$ together contain
representatives of all the conjugacy classes in $PGL(2,q)$, and we will exploit this observation.

We now consider the structure of the subgroups $H_1$ and $H_2$. Since the $H_i$ are point stabilizers, Lemma 2 applies. We deduce that $H_2$ is cyclic. Thus, any element
of $PGL(2,q)$ conjugate to an element of $H_2$ is centralized by a cyclic subgroup of order
$q+1$ and hence the order of its centralizer is divisible by $q+1$. The subgroup
$H_1$ has a normal subgroup $Q$, say, of order $q$ with cyclic quotient of order $q-1$. It is the normalizer of a Sylow $p$-subgroup of $PGL(2,q)$. Note that the $H_i$ both have even order when $q$ is odd.

Consider a non-identity element $h$, say, of $H_1$. If its order is divisible by $p$, Lemma 3 implies that its order is exactly $p$ and its centralizer has order $q$. If its order is
coprime to $p$, it is a consequence of the Schur-Zassenhaus theorem, for example, that
$h$ is contained in a subgroup of $H_1$ of order $q-1$. This subgroup is cyclic by Lemma
2 and hence the centralizer of $h$ in $PGL(2,q)$ has order divisible by $q-1$.

We summarize this discussion in the following lemma.

\begin{lemma} \label{centralizer_order_estimate}
Let $s$ be a non-identity element of $PGL(2,q)$. Then the order of the centralizer
of $s$ in $PGL(2,q)$ is divisible by one of the integers $q-1$, $q$ or $q+1$. Thus, the centralizer of $s$ has order at least $q-1$.
\end{lemma}

We will proceed to use these ideas to discuss 
the conjugacy classes involutions (elements of order 2) in $G=PGL(2,q)$. 

\begin{lemma} \label{conjugacy_of_involutions}
Suppose that $q$ is odd. Then there are exactly two conjugacy classes of involutions
in $G=PGL(2,q)$. The involutions in one conjugacy class fix exactly two elements
of $\mathbb{P}^1(\mathbb{F}_q)$, and the involutions in the other class fix exactly
two elements of $\mathbb{F}_{q^2}\setminus \mathbb{F}_q$.
\end{lemma}

\begin{proof}
As we showed above, each element of $G$ is conjugate to an element of one of the subgroups
$H_1$ and $H_2$. Since we are assuming that $q$ is odd, the $H_i$ both have a cyclic
Sylow 2-subgroup (but of different orders). Now a cyclic 2-group clearly has a unique
involution and it follows from Sylow's theorem that each of the subgroups $H_i$
has a single conjugacy class of involutions. This in turn implies that $G$ has at most two
conjugacy classes of involutions.

We now show $G$ has precisely two conjugacy classes of involutions, and distinguish between
the classes. Let $t_1$ be an involution in $H_1$. Then $t_1$ fixes a point of $\mathbb{P}^1(\mathbb{F}_q)$. Since $t_1$ has even order, and $|\mathbb{P}^1(\mathbb{F}_q)|=q+1$ is also even, it is trivial to prove that $t_1$ fixes an even number of elements of
$\mathbb{P}^1(\mathbb{F}_q)$. Thus $t_1$ fixes at least two points of $\mathbb{P}^1(\mathbb{F}_q)$. Lemma \ref{two_fixed} implies that $t_1$ fixes no additional points
in its action on the projective line over the algebraic closure, and hence, certainly
no point of $\mathbb{F}_{q^2}\setminus \mathbb{F}_q$.

Let $t_2$ be an involution in $H_2$. The same argument as above shows that $t_2$ fixes exactly two points of $\mathbb{F}_{q^2}\setminus \mathbb{F}_q$ and none of $\mathbb{P}^1(\mathbb{F}_q)$. This implies that $t_1$ and $t_2$ are not conjugate in $G$, since otherwise $t_2$ would fix a point of $\mathbb{P}^1(\mathbb{F}_q)$. Thus there are precisely
two conjugacy classes of involutions.
\end{proof}

The analysis of involutions when $q$ is a power of 2 is easier, since it follows from the proof of Lemma
\ref{Sylow} (section 3) that in this case, there is a single class
of these elements in $PGL(2,q)$.

We have seen the significance of equations of the form $s(\alpha)=\alpha^q$, where $s$ is an element of $PGL(2,q)$ and $\alpha$ is in $\overline{\mathbb{F}_q}$, and we examine this
type of equation in more detail here.

\begin{lemma} \label{number_of_solutions_of_equation}
Let $s$ be an element of $PGL(2,q)$. Then the number of elements $\alpha$ in
$\overline{\mathbb{F}_q}$ that satisfy $s(\alpha)=\alpha^q$ is either $q$ or $q+1$. 
\end{lemma}

\begin{proof}
We express the action of $s$ in fractional form $s(x)=(ax+b)/(cx+d)$. Then $\alpha$ in
$\overline{\mathbb{F}}_q$ satisfies $s(\alpha)=\alpha^q$ precisely when
\[
c\alpha^{q+1}+d\alpha^q-a\alpha-b=0.
\]
Since $c$ and $d$ cannot both be 0, $\alpha$ is a root of a polynomial of degree 
$q$ or $q+1$. Moreover, this polynomial has distinct roots.  This proves what we require.
\end{proof}

The next lemma provides a little more detail about the solutions of $s(\alpha)=\alpha^q$ provided $s$ has order greater than 2.

\begin{lemma} \label{solutions_in_regular_orbits}
Let $s$ be an element of $PGL(2,q)$ of order greater than $2$. Then there exists some $\alpha$ in
$\overline{\mathbb{F}_q}\setminus \mathbb{F}_{q^2}$ such that $s(\alpha)=\alpha^q$. Thus $\alpha$ is in a regular orbit of $PGL(2,q)$ acting on $\mathbb{P}^1(\overline{\mathbb{F}_q})$.
\end{lemma}

\begin{proof}
Suppose that all solutions of $s(z)=z^q$ lie in $\mathbb{F}_{q^2}$. Then any solution
$\alpha$ satisfies $s^2(\alpha)=\alpha$. Now by assumption on the order of $s$, $s^2$ is not the identity and hence it has at most two fixed points on $\mathbb{P}^1(\overline{\mathbb{F}_q})$ by Lemma \ref{two_fixed}. On the other hand, there are at least $q$ solutions to
$s(z)=z^q$ by Lemma \ref{number_of_solutions_of_equation} and hence at least $q$ different
fixed points of $s^2$. We have a contradiction unless $q=2$. However, in the exceptional case when $q=2$, $s$ has order 3 in $PGL(2,q)$ and the equation $s(z)=z^q$ is easily seen to have three different solutions which all lie in $\mathbb{F}_8$ but not in $\mathbb{F}_2$.
This is another contradiction, and our lemma is established.
\end{proof}

The next lemma, also concerned with the equation $s(\alpha)=\alpha^q$, shows how  solutions in a regular orbit are derived from one given solution. 

\begin{lemma} \label{action_of_centralizer_on_roots}
Let $s$ be an element of $G=PGL(2,q)$ of order greater than $2$ and let $\alpha$ be an element of $\overline{\mathbb{F}_q}$ that, in accordance with Lemma $\ref{solutions_in_regular_orbits}$, satisfies
$s(\alpha)=\alpha^q$ and lies in a regular orbit of $G$ acting on
$\mathbb{P}^1(\overline{\mathbb{F}_q})$. Then any other element $\beta$ of 
$\overline{\mathbb{F}_q}$ that satisfies $s(\beta)=\beta^q$ and lies in a regular $G$-orbit has the form $\beta=c(\alpha)$,
where $c\in G$ centralizes $s$. Thus $\alpha$ and $\beta$ are in the same (regular)
$G$-orbit.
\end{lemma}

\begin{proof}
Let $C$ be the centralizer of $s$ in $G$. As we have described in the discussion above, 
$|C|$ is divisible by $q-1$, $q$ or $q+1$. Let $c$ be an element of $C$. Then we claim that
$c(\alpha)$ is also a solution of $s(z)=z^q$. 
For, since $s$ commutes with $c$ and the $G$-action commutes with Frobenius action,
\[
s(c(\alpha))=c(s(\alpha))=c(\alpha^q)=c(\alpha)^q.
\]
Thus $c(\alpha)$ has the desired property and since $\alpha$ is in a regular orbit,
we obtain at least $q-1$ solutions to $s(z)=z^q$ which all lie in the same $G$-orbit.

Suppose if possible that we have another element $\beta$ that is a solution of
$s(z)=z^q$ and lies in a different regular $G$-orbit. Then we obtain at least $q-1$  further solutions to the equation $s(z)=z^q$.
Since these solutions lie in a different  $G$-orbit to $\alpha$,
we have at least $2(q-1)$ solutions to the equation $s(z)=z^q$.
Lemma \ref{number_of_solutions_of_equation} implies that $2(q-1)\leq q+1$. 
This can only occur when $q=2$ or $3$.

We consider these two values of $q$ separately. When $q=2$, the only possibility
for $s$ that  arises is that it is an element of order $3$, 
with centralizer of order $3$, and clearly we have a contradiction here. 
 When $q=3$, the two conjugacy classes of $PGL(2,3)$ that we are considering (those that are not involutions or the identity) consist of elements of order
$3$ and $4$, respectively, and the corresponding centralizers have order $3$ and $4$, not 
$2=q-1$. Thus, even when $q=3$, we have a contradiction and have thus proved that all
solutions of $s(z)=z^q$ which lie in regular $G$-orbits are in the same regular orbit.

Finally, suppose that $\gamma$ and $\delta$ are elements of $\overline{\mathbb{F}_q}$
that lie in the same regular $G$-orbit and satisfy
\[
s(\gamma)=\gamma^q,\qquad s(\delta)=\delta^q.
\]
Then since $\gamma$ and $\delta$ lie in the same regular $G$-orbit, $\gamma=g(\delta)$
for a unique $g$ in $G$. We then have $\gamma^q=g(\delta^q)$ and hence
\[
s(\gamma)=\gamma^q=g(\delta^q)=gs(\delta)=gsg^{-1}(\gamma).
\]
Since $\gamma$ belongs to a regular $G$-orbit, it must be the case that $s=gsg^{-1}$, and hence $g$ centralizes $s$. We have thus proved that the solutions of $s(z)=z^q$ in the same regular orbit have the form
$c(\alpha)$ where $\alpha$ is a given solution and $c$ centralizes $s$, as required.
\end{proof}

Note in connection with Lemma \ref{action_of_centralizer_on_roots} that we are not claiming that all solutions of $s(z)=z^q$ are in regular $G$-orbits. For it is possible
that $s$ fixes an element $\alpha$, say, of $\mathbb{F}_q$. Then we have $s(\alpha)=\alpha^q$ but $\alpha$ is certainly not in a regular orbit. 

We conclude this section with two special results relating to involutions.

\begin{lemma} \label{action_of_an_involution}
Let $s$ be an involution in $PGL(2,q)$. Then there exists $\alpha$ in
$\mathbb{F}_{q^2}\setminus \mathbb{F}_q$ such that $s(\alpha)=\alpha^q$.
\end{lemma}

\begin{proof}
Lemma \ref{number_of_solutions_of_equation} shows that there at least $q$ solutions in
$\overline{\mathbb{F}_q}$ of the equation $s(z)=z^q$. Suppose if possible that all these solutions lie in $\mathbb{F}_q$. Then it follows that $s$ fixes at least $q$ different points of $\mathbb{F}_q$. Lemma \ref{two_fixed} implies that this is only possible if
$q=2$. But now, in this exceptional case, Lemma \ref{two_fixed} also implies that as we are in characteristic 2 and $s$ is an involution, it fixes only one point, and we have a contradiction. Thus the required $\alpha$ exists.
\end{proof}

\begin{lemma} \label{it_is_an_involution}
Let $\alpha$ be an element of $\mathbb{F}_{q^2}\setminus \mathbb{F}_q$ and let $s$ be an element of $PGL(2,q)$ that satisfies $s(\alpha)=\alpha^q$. Then $s$ is an involution.
\end{lemma}

\begin{proof}
We have
\[
s^2(\alpha)=\alpha^{q^2}=\alpha,
\]
since $\alpha$ is in $\mathbb{F}_{q^2}$. Suppose that $s^2$ is not the identity. Then Lemma
\ref{powersfix} implies that $s$ fixes $\alpha$. This means that $\alpha$ is in
$\mathbb{F}_q$, contrary to hypothesis. It must the case that $s^2=1$, and $s$ is indeed an involution.
\end{proof}

\subsection{Conjugacy Classes and Regular Orbits}

We show here how to use the previous three lemmas to link points in $\mathbb{F}_q$ with conjugacy classes of elements of order greater than 2 in $PGL(2,q)$.

\begin{thm} \label{conjugacy_class_correspondence}
Let $G=PGL(2,q)$ and let $\Phi$ be a generator of $\mathbb{F}_q(x)^G$. 
Let $\gamma$ be any element of $\mathbb{F}_{q^2}\setminus \mathbb{F}_q$ and let
$\mu=\Phi(\gamma)$. Then $\mu$ is an element of $\mathbb{F}_q$. Let
$\lambda$ be any element of $\mathbb{F}_q$ different from $\mu$ and 
let $\alpha\in \overline{\mathbb{F}_q}$
satisfy $\Phi(\alpha)=\lambda$. Let $s\in G$ satisfy $s(\alpha)=\alpha^q$, in accordance
with Lemma \ref{main3a}. 
Then $\alpha$ lies in a regular orbit of $G$ acting on
$\mathbb{P}^1(\overline{\mathbb{F}_q})$ 
and $\lambda$ determines a unique conjugacy class of
$G$, namely, the conjugacy class of $s$. Each conjugacy class of $G$ consisting of elements
of order greater than $2$ arises in this correspondence exactly once.
\end{thm}

\begin{proof}
We note first that, by Corollary \ref{fq23}, $\Phi(\gamma)=\mu$ is an element of $\mathbb{F}_q$.
Now consider for $\lambda\neq \mu$, any element $\alpha$ in $\overline{\mathbb{F}_q}$
that satisfies $\Phi(\alpha)=\lambda$. We claim that $\alpha$ is not in $\mathbb{F}_q$ or
in $\mathbb{F}_{q^2}$. For the two non-regular $G$-orbits are $\mathbb{P}^1(\mathbb{F}_q)$
and $\mathbb{F}_{q^2}\setminus \mathbb{F}_q$. We have $\Phi(\beta)=\infty$ for
all points of $\mathbb{F}_q$, whereas, since $\mathbb{F}_{q^2}\setminus \mathbb{F}_q$
is a single $G$-orbit,
\[
\Phi(\beta)=\Phi(\gamma)=\mu
\]
if $\beta\in \mathbb{F}_{q^2}\setminus \mathbb{F}_q$.

 Thus, we have proved that
$\alpha$ is not in $\mathbb{F}_q$ nor
in $\mathbb{F}_{q^2}$, and is thus in a regular $G$-orbit. It follows then
that $s$ is the unique element
of $G$ mapping $\alpha$ into $\alpha^q$. 

Suppose now that 
$\beta\in \overline{\mathbb{F}_q}$ also satisfies $\Phi(\beta)=\lambda$. 
By Lemma \ref{Phiorbit}, $\beta$ is in the same $G$-orbit as $\alpha$, so $\beta$ is in a regular orbit.
Then there is a unique element $t$ in $G$ with $t(\beta)=\beta^q$. 
We aim to show that $t$ is conjugate in $G$ to $s$. 

Because $\alpha$ and $\beta$ both determine the same $\Phi$ value, Lemma 10 implies that there is an element $g$ in $G$ with $\beta=g(\alpha)$. 
The $G$-action on $\mathbb{P}^1(\overline{\mathbb{F}_q})$ commutes with the Frobenius action, 
as shown in the proof of Theorem \ref{falpha1}.
It follows that
\[
gsg^{-1}(\beta)=gs(\alpha)=g(\alpha^q)=g(\alpha)^q=\beta^q.
\]
By uniqueness of $t$, we get $gsg^{-1}=t$ and $t$ is conjugate to $s$, as required.

We now show that the conjugacy class of $s$ is not associated to any other element of
$\mathbb{F}_q$ different from $\mu$ and $\lambda$ just considered. For suppose that
there exists $\delta$ in $\overline{\mathbb{F}_q}$ such that $\Phi(\delta)=\kappa$, where
$\kappa$ is different from $\mu$ and $\lambda$, and $s(\delta)=\delta^q$. Then, as we argued
above, $\delta$ is in a regular $G$-orbit. 

Next we apply Lemma \ref{action_of_centralizer_on_roots}. Since $\alpha$ and $\delta$ are both in regular $G$-orbits and are solutions to $s(z)=z^q$, they lie in the same
$G$-orbit. But this implies that $\Phi$ takes the same value on $\alpha$ and $\delta$, contrary to assumption. It follows that the conjugacy class of $s$ is uniquely associated
to the element $\lambda$ of $\mathbb{F}_q$, as required.

We conclude by showing that if $t$ is any element of $G$ of order greater than
2, its conjugacy class in $G$ arises in this association. Lemma \ref{solutions_in_regular_orbits} shows that there is some element $\epsilon$ in
$\overline{\mathbb{F}_q}$ that lies in a regular $G$-orbit and satisfies $t(\epsilon)=\epsilon^q$. Then $\Phi(\epsilon)=\Phi(\epsilon)^q$ and thus $\Phi(\epsilon)$ is an element
$\nu$, say, of $\mathbb{F}_q$. The conjugacy class of $t$ in $G$ is therefore associated
with $\nu$, and we have thus  shown 
that all conjugacy classes of elements of order
greater than 2 arise uniquely.
\end{proof}

\subsection{Non-regular Orbits}

We have set $\Phi(\gamma)=\mu$, where $\gamma$ is any element of $\mathbb{F}_{q^2}\setminus \mathbb{F}_q$. We investigate here if $\mu$ determines a unique conjugacy class of
$G=PGL(2,q)$. As we have seen, this involves solving the equation $s(\gamma)=\gamma^q$
for $s$ in $G$. Lemma \ref{it_is_an_involution} shows that $s$ is an involution.
Thus, from our discussion of involutions, the conjugacy class of $s$ is uniquely
determined by $\mu$ when $q$ is even, and consequently we have a complete correspondence between
conjugacy classes of $G$ and values of $\Phi$ in $\mathbb{F}_q$.

On the other hand, suppose that $q$ is odd, and let $s$ and $t$ be representatives
of the two different conjugacy classes of involutions in $G$. Lemma \ref{action_of_an_involution} shows that there are elements $\alpha$ and $\beta$ of 
$\mathbb{F}_{q^2}\setminus \mathbb{F}_q$ with
\[
s(\alpha)=\alpha^q,\qquad t(\beta)=\beta^q,
\]
and of course we have $\Phi(\alpha)=\Phi(\beta)=\mu$.

Thus $\mu$ does not determine a unique conjugacy class of $G$, and indeed both conjugacy classes of involutions are associated to $\mu$.
We deduce that $PGL(2,q)$ has $q+2$ conjugacy classes when $q$ is odd.

\subsection{Final Summary}

We summarize the results of this section here, which provide an extension 
of Theorem \ref{falpha1}.
 Associating the point $\infty$ of $\mathbb{P}^1(\mathbb{F}_q)$ with the identity class
of $PGL(2,q)$, we have established the following correspondence when $q$ is even.

\begin{thm} \label{conjugacy_classes_and_projective_points}
Let $q$ be a power of $2$ and let $G=PGL(2,q)$. Let $\Phi=f(x)/g(x)$ be a generator of 
$\mathbb{F}_q(x)^G$  with $\deg (f)=|G|$ and $\deg (g) < |G|$.
There is a 1-1 correspondence between
the points of $\mathbb{P}^1(\F_q)$ and the conjugacy classes of $G$. 
In this correspondence, $\infty$ corresponds to the identity class. For any element
$\lambda$ of $\mathbb{F}_q$, let $\alpha$ in $\overline{\mathbb{F}_q}$ satisfy
$\Phi(\alpha)=\lambda$ and let $s$ be an element of $G$ with
$s(\alpha)=\alpha^q$. Then $\lambda$ corresponds to the conjugacy class of $s$ in $G$. 

Furthermore, provided $\alpha$ lies in a regular orbit of $G$ acting on 
$\mathbb{P}^1(\overline{\mathbb{F}_q})$, the polynomial $F_\alpha(T)=f(T)-\lambda g(T)$
is a product of $|G|/r$ different irreducible polynomials of degree $r$, where
$r$ is the order of $s$. In the case $\alpha$ does not lie in a regular orbit (which occurs
when $\alpha$ is in $\mathbb{F}_{q^2}\setminus
\mathbb{F}_q$), $F_\alpha(T)$ is a product of $(q^2-q)/2$ different irreducible  polynomials of degree $2$, each occurring with multiplicity $q+1$.
\end{thm}

When $q$ is odd we have the following summary.

\begin{thm} \label{conjugacy_classes_in_odd_characteristic}
Let $q$ be a power of an odd prime and let $G=PGL(2,q)$. Let $\Phi=f(x)/g(x)$ be a generator of 
$\mathbb{F}_q(x)^G$ with $\deg (f)=|G|$ and $\deg (g) < |G|$,
There is a 1-1 correspondence between $q$
 points of $\mathbb{P}^1(\F_q)$ and $q$ of the conjugacy classes of $G$. 
In this correspondence, $\infty$ corresponds to the identity class. For any element
$\lambda$ of $\mathbb{F}_q$, let $\alpha$ in $\overline{\mathbb{F}_q}$ satisfy $\Phi(\alpha)=\lambda$ and let $s$ be an element of $G$ that satisfies
$s(\alpha)=\alpha^q$. Then $\lambda$ corresponds uniquely to the conjugacy class of $s$ in $G$ except
in the case that $\alpha$ is in $\mathbb{F}_{q^2}\setminus \mathbb{F}_q$. 
In the case that $\alpha\in \mathbb{F}_{q^2}\setminus \mathbb{F}_q$, 
$s$ is an involution in $G$ but its conjugacy class is not uniquely determined.
\end{thm}

\section{Subgroups of $PGL(2,F)$ and $S_5$}\label{s5}

The subgroups of $PGL(2,F)$ have been classified. 
This was first done by Klein in the case $F=\mathbb{C}$.
It is not our intention here to state or discuss the classification;
for further details we refer the reader to \cite{HKT} (Theorem 11.91) for example.
In summary, the possible subgroups are $PSL(2,F)$, the subfield subgroups $PGL(2,E)$
where $E$ is a subfield of $F$,
cyclic groups, dihedral groups, $A_4$, $S_4$, $A_5$,
elementary abelian groups or semi-direct products of elementary abelian groups with cyclic groups.

We wish to present an elementary proof that 
$PGL(2,F)$ never contains $S_5$ as a subgroup, when $F$ is a field
of characteristic different from $5$.
If $F$ has characteristic 5 then the opposite is true:
since  $PGL(2,5)$ is isomorphic to $S_5$, 
clearly  $S_5$ appears as a subfield subgroup of $PGL(2,F)$.

Now we present the proof, which uses lemmas proved in Section \ref{pglactp1}.

\begin{thm} \label{no_S_5}
Let $F$ be a field of characteristic different from $5$. Then $S_5$ is not a subgroup of $PGL(2,F)$. 
\end{thm}

\begin{proof}
We recall that the proof of Lemma \ref{powersfix}
shows that if $p$ is a prime number, 
then an element of order  divisible by $p$ in $PGL(2,F)$ in fact has order equal to  $p$.
  
Suppose first that $F$ has characteristic 2.  
Since $S_5$ contains an element of order 4, it cannot be a subgroup of $PGL(2,F)$.

Next we suppose that $F$ has characteristic 3. 
By the first paragraph, an
element of order divisible by 3 in $PGL(2,F)$ in fact has order 3. 
Since $S_5$ contains an element of order 6, it cannot be a subgroup of $PGL(2,F)$.

For the rest of the proof, it suffices to assume that $F$ is algebraically closed of characteristic different from 2, 3 or 5. Suppose for the sake of contradiction that $S_5$ is a subgroup of
$PGL(2,F)$. Certainly, $S_5$ acts on the projective line $\mathbb{P}^1(F)$ 
and some of its orbits are not regular, as each of its elements is contained in some point stabilizer. Furthermore, $S_5$ does not act with exactly three non-regular orbits. For, it has no abelian subgroup of index 2 and its order is greater than 60, and it is thus excluded by Lemma \ref{three_non_regular_orbits}.

We now consider an action of $S_5$ on $\mathbb{P}^1(F)$ with exactly two non-regular orbits. Let
$\Omega_1$ and $\Omega_2$ be these two orbits, with corresponding point stabilizers
$G_1$ and $G_2$. Lemma \ref{cyclic_stabilizer} implies that $G_1$ and $G_2$ are cyclic. Furthermore, as each element of $S_5$ fixes two points by Lemma \ref{two_fixed}, 
$G_1$ and $G_2$ contain elements of order 3, 4 and 5. It follows that at least one of $|G_1|$ and $|G_2|$ is divisible  by 12, 15 or 20, and as each subgroup is cyclic, contains an element of order 12, 15 or 20.
However, $S_5$ contains no elements of these orders. Similarly, $S_5$ cannot act with only one non-regular orbit, as a point stabilizer would have to have order divisible by 60 and this is impossible. It follows that $S_5$ is not a subgroup of $PGL(2,F)$. 
\end{proof}

We remark that characteristic 3 is an interesting case, 
because $PGL(2,9)$ has order 720, which is the order of $S_6$, 
and contains a subgroup of index 2 isomorphic to $A_6$. Of course, 
 $PGL(2,9)$ cannot be isomorphic to $S_6$, since if it were then $S_5$ would be realizable
 as a subgroup of $PGL(2,9)$.

\section{Lang's Theorem}\label{lang}

In this section we will present an alternative proof of Theorem \ref{main3b}.
This proof uses a theorem of Lang on algebraic groups, and provides a different
context for Theorem \ref{main3b}.

Let $K$ denote the algebraic closure of $\F_q$.
Let $\sigma$ denote the Frobenius automorphism of $K$, given by $\sigma(x)=x^q$.
Then $\sigma$ acts naturally on $GL(2,K)$ by mapping
${a \ b\choose c\ d}$ to ${a^q \ b^q\choose c^q\ d^q}$.
As $GL(2,K)$ is a connected algebraic group, Lang's theorem \cite{L} implies that
every element $g\in GL(2,K)$ is expressible in the form
$g=\sigma(h)^{-1}h$ for some $h\in GL(2,K)$.

Let $\pi$ denote the natural projection homomorphism from $GL(2,K)$ to
$PGL(2,K)$. Applying $\pi$ to the equation $g=\sigma(h)^{-1}h$
implies that every element $x\in PGL(2,K)$ is expressible in the form
$x=\sigma(y)^{-1}y$ for some $y\in PGL(2,K)$.

Let $X=\mathbb{P}^1(K)$ denote the projective line over $K$. We let $\sigma$ act on $X$ in the obvious way, sending $K$ into itself and mapping $\infty$ to itself. 
Let $X^\sigma$ denote the fixed point set of $\sigma$ acting on $X$. 
We can identify $X^\sigma$ with the 
projective line over $\mathbb{F}_q$.
Likewise, let  $G=PGL(2,K)$ and let $G^\sigma$ be the fixed point subgroup of $\sigma$ acting on $G$. 
We can identify $G^\sigma$ with $PGL(2,q)$. 
Note that $G$ acts on $X$ and we have a compatibility relation with the action of $\sigma$:
\[
\sigma(g\alpha)=\sigma(g)\sigma(\alpha)
\]
for $g\in G$ and $\alpha$ in $X$.

 {\bf Alternative Proof of Theorem \ref{main3b}}
 
 The essential part of the proof (in part 1) of Theorem \ref{main3b} is to show
 that if $s\in PGL(2,q)$ then a solution $\alpha$ to $s(\alpha)=\alpha^q$ lies in 
 $\F_{q^r}$ and no smaller field if and only if $s$ has order $r$.
 Here we obtain this result and in addition we obtain a neat description of the solutions.

 \begin{lemma}
 Let $G^\sigma=PGL(2,q)$ and let $s$ be any element of $G^\sigma$  of order $r$.
 Then the solutions of $s(x)=\sigma(x)$ in $X$ lie in $\mathbb{P}^1 (\F_{q^r})$
 on the image of  $X^\sigma$ under an element of $PGL(2,\mathbb{F}_{q^r})$.
 Also, $G^\sigma$ is conjugate to a subgroup $H$ of $PGL(2,\mathbb{F}_{q^r})$
 via the same element, and $H$ acts on the solutions of $s(x)=\sigma(x)$ as
 $G^\sigma$ acts on $X^\sigma$.
 \end{lemma}
 
 \begin{proof}
 Let $s$ be any element of $G^\sigma=PGL(2,q)$ and by Lang's theorem 
 write $s=\sigma(t)^{-1}t$ for $t$ in $G$. 
 Set $H=t^{-1}G^\sigma t$, a conjugate subgroup of $G^\sigma$ in $G$. Let
$h=t^{-1}gt$ be an element of $H$, where $g$ is in $G^\sigma$. We have then 
\[
\sigma(h)=\sigma(t)^{-1}g\sigma(t)=st^{-1}gts^{-1}=shs^{-1}.
\]
Assume  that $s$ has order $r$. Then
\[
\sigma^r(h)=s^rhs^{-r}=h.
\]
Thus all elements of $H$ are fixed by $\sigma^r$ and hence $H$ is a 
subgroup of $PGL(2,\mathbb{F}_{q^r})$.

Concerning the element $t$, write $\sigma(t)=tu$, where $u =s^{-1}$. 
Apply $\sigma$ successively to get $\sigma^r(t)=tu^r$. 
If $s$ and hence $u$ have order $r$, then
$\sigma^r$ fixes $t$ and $t$ must be in $PGL(2,q^r)$.

Let  $X_s=\{x\in X: s(x)=\sigma(x)\}$. Note that if $\alpha$ is in $X_s$, then
$s^r(\alpha)=\sigma^r(\alpha)$.
If $s$ has order $r$, then $X_s$ is in the projective line over $\mathbb{F}_{q^r}$. 
As before, we set $s=\sigma(t)^{-1}t$ by Lang's theorem, and we find that
$s(\alpha)=\sigma(\alpha)$ is equivalent to
\[
t(\alpha)=\sigma(t(\alpha)).
\]
Thus $t(\alpha)$ is in $X^\sigma$ and $\alpha$ is in $t^{-1}(X^\sigma)$. 
It follows  that $X_s=t^{-1}(X^\sigma)$, as required.
Furthermore, $H=t^{-1}G^\sigma t$ acts on $X_s$ as the projective group of the line $X^\sigma$. 
\end{proof}


\begin{thebibliography}{99}

\bibitem{Ab} 
Abhyankar, S.S.
Galois theory on the line in nonzero characteristic. Bull. Amer. Math. Soc. {\bf 27} (1992), 68-133.

\bibitem{A} Artin, E.
Galois Theory: Lectures delivered at the University of Notre Dame,
Notre Dame Mathematical Lectures, Number 2, 1944.

\bibitem{Bosma} 
Bosma, W.; Cannon, J.; Playoust, C.
 The Magma algebra system. I. The user language.
  J. Symbolic Comput. {\bf 24} (1997), 235-265.


\bibitem{B} Bluher, A.
Explicit Artin maps into $PGL_2$. arXiv:1906.08944v3, 1 Apr. 2021.

\bibitem{CKT} Chu, H.;  Kang, M. C.; Tan, E.T.
The invariants of projective linear group actions.
Bull. Austral. Math. Soc. {\bf 39} (1989), 107-117.


\bibitem{D}
Dickson, L. E. An invariantive investigation of irreducible binary modular forms. Trans.
Amer. Math. Soc. {\bf 12} (1911), 1-8.


\bibitem{F} Foster, K.
HT90 and "simplest" number fields.
Illinois J. Math. {\bf 55} (2011), 1621-1655.

\bibitem{GS} Gutierrez, J;  Sevilla, D.
On Ritt's decomposition theorem in the case of finite fields.
Finite Fields Appl. {\bf 12} (2006), 403-412

\bibitem{HKT} Hirschfeld, J.W.P.; Korchm\'aros, G.; Torres, F.
Algebraic curves over a finite field.
Princeton University Press, Princeton, NJ, 2008.


\bibitem{L}
Lang, S. . Algebraic Groups Over Finite Fields. American Journal of Mathematics, 78(3), (1956), 555-563. doi:10.2307/2372673


\bibitem{PRW} Panario, D.; Reis, L.; Wang, Q.
Construction of irreducible polynomials through rational transformations.
J. Pure Appl. Math. {\bf 224} (2020), 106241.

\bibitem{R}
Rivoire, P. Fonctions rationnelles sur un corps fini. Ann. Inst. Fourier {\bf 6} (1955/6),  121-124

\bibitem{ST} Stichtenoth, H.;  Topuzo\u{g}lu, A.
Factorization of a class of polynomials over finite fields.
Finite Fields Appl. {\bf 18} (2012), 108-122

\bibitem{vdw} Waerden, B.L. van der.
Modern Algebra. Vol. 1. Frederick Ungar Publishing Co., New York, NY, 1949.


\end{thebibliography}
\end{document}